\documentclass[10pt]{amsart}
\usepackage{amssymb}
\usepackage{dsfont}
\usepackage{amscd}
\usepackage[mathscr]{euscript} % for the power-set P
\usepackage[all]{xy}
\usepackage{url}
\usepackage{stmaryrd}
\usepackage{comment}
\usepackage[retainorgcmds]{IEEEtrantools}
\usepackage{hyperref}
\usepackage{graphicx}

\title{On existence of canonical ${G}$-bases}
\author[D. HOFFMANN]{Daniel Max Hoffmann$^{\dagger}$}
\thanks{$^{\dagger}$SDG. Part of this work was conducted during author's internship at the Warsaw Center of Mathematics and Computer Science. Partially supported by NCN grant 2016/20/T/ST1/00482 and NCN grant 2015/19/B/ST1/01150}
\address{$^{\dagger}$Instytut Matematyczny\\
Uniwersytet Wroc{\l}awski\\
Wroc{\l}aw\\
Poland}
\email{daniel.hoffmann@math.uni.wroc.pl}

\DeclareMathOperator{\Der}{Der}

  \DeclareMathOperator{\id}{id}
 \DeclareMathOperator{\fr}{Fr}

\DeclareMathOperator{\im}{im}  
  
 \DeclareMathOperator{\spe}{Spec}

\DeclareMathOperator{\ev}{ev}

\DeclareMathOperator{\ddf}{DF}\DeclareMathOperator{\dcf}{DCF}
\newtheorem{theorem}{Theorem}[section]
\newtheorem{prop}[theorem]{Proposition}
\newtheorem{lemma}[theorem]{Lemma}
\newtheorem{cor}[theorem]{Corollary}
\newtheorem{fact}[theorem]{Fact}
\theoremstyle{definition}
\newtheorem{definition}[theorem]{Definition}
\newtheorem{example}[theorem]{Example}
\newtheorem{remark}[theorem]{Remark}

\theoremstyle{remark}

\makeatletter
\providecommand*{\cupdot}{%
  \mathbin{%
    \mathpalette\@cupdot{}%
  }%
}
\newcommand*{\@cupdot}[2]{%
  \ooalign{%
    $\m@th#1\cup$\cr
    \hidewidth$\m@th#1\cdot$\hidewidth
  }%
}
\makeatother

\begin{document}

\newcommand{\twoc}[3]{ {#1} \choose {{#2}|{#3}}}
\newcommand{\thrc}[4]{ {#1} \choose {{#2}|{#3}|{#4}}}
\newcommand{\Zz}{{\mathds{Z}}}
\newcommand{\Ff}{{\mathds{F}}}
\newcommand{\Cc}{{\mathds{C}}}
\newcommand{\Rr}{{\mathds{R}}}
\newcommand{\Nn}{{\mathds{N}}}
\newcommand{\Qq}{{\mathds{Q}}}
\newcommand{\Kk}{{\mathds{K}}}
\newcommand{\Pp}{{\mathds{P}}}
\newcommand{\ddd}{\mathrm{d}}
\newcommand{\Aa}{\mathds{A}}
\newcommand{\dlog}{\mathrm{ld}}
\newcommand{\ga}{\mathbb{G}_{\rm{a}}}
\newcommand{\gm}{\mathbb{G}_{\rm{m}}}
\newcommand{\gaf}{\widehat{\mathbb{G}}_{\rm{a}}}
\newcommand{\gmf}{\widehat{\mathbb{G}}_{\rm{m}}}
\newcommand{\gdf}{\mathfrak{g}-\ddf}
\newcommand{\gdcf}{\mathfrak{g}-\dcf}
\newcommand{\fdf}{F-\ddf}
\newcommand{\fdcf}{F-\dcf}

\maketitle
\begin{abstract}
We describe a general method for expanding a truncated $G$-iterative Hasse-Schmidt derivation, where $G$ is an algebraic group.
We give examples of algebraic groups for which our method works.
\end{abstract}

\tableofcontents

\section{Introduction}
\noindent
Our motivation for this paper is \cite{Zieg3} and \cite{K3}, where some nice model-theoretic properties are obtained for fields
equipped with HS-derivations satisfying the standard iterativity rule.
Analyzing the reasoning in \cite{Zieg3} and \cite{K3}, we deduce that one of the most important properties of an iterative Hasse-Schmidt derivation is 
Matsumura's \emph{strong integrability} (a notion from \cite{Mats1}, see: Definition \ref{defintegrable}). Thus we are especially interested in it.
\\
Briefly, strong integrability means that a truncated iterative HS-derivation can be expanded to a not-truncated one,
satisfying the same iterativity conditions. We prove (Theorem \ref{integrable}) that the existence of a \emph{canonical basis} (Definition
\ref{defbasis}) implies strong integrability for an arbitrary iterativity condition. However, the converse is not true in general
(see Remark \ref{remintg}),
which is related to the problem of the existence of canonical basis in a given field.
\\
Finding a canonical basis is not an easy task. Hideyuki Matsumura
in \cite{Mats1} proved the existence of canonical basis for $\mathbb{G}_a$ (the standard iterativity).
Afterwards Andrzej Tyc in \cite{Tyc}  did the same for $\mathbb{G}_m$ and one dimensional formal groups over
algebraically closed fields. Martin Ziegler showed existence of canonical bases for powers of $\mathbb{G}_a$ proving
the quantifier elimination for the theory of separably closed fields in \cite{Zieg3} and \cite{Zieg2} (see Example \ref{ZieglerEx}).
Before this paper only products of $\mathbb{G}_a$ and $\mathbb{G}_m$ were considered. We cover the case of
commutative, connected, unipotent groups of dimension $2$ over an algebraically closed field. 
That leads us to Theorem \ref{thmfinal}, stating that, over an algebraically closed field, linear algebraic groups
that are connected and commutative have canonical basis if unipotent elements form a subgroup of dimension $\le 2$.
This theorem includes all the previous results (mentioned above).
\\
Piotr Kowalski and I in \cite{HK} are treating iterative HS-derivations in much more abstract way. Many proofs from \cite{HK} would
be obvious if canonical bases exist for the HS-derivations considered there (a similar sentence was noted at the end of Section 2. in \cite{K3}).
Moreover, Section 6. in \cite{HK} suggests possible generalisations for the notion of canonical basis.

I thank my advisor, Piotr Kowalski, for his invaluable contribution.
I am also grateful to the referee for a very useful report.

\section{Basic notions about $F$-derivations}\label{secsetup}
\subsection{HS-derivations}
\noindent
All the rings considered in this paper are commutative and with unity.
Fix a field $k$ of the characteristic $p>0$, $e\in\mathbb{N}_{>0}$ and $m\in\mathbb{N}_{>0}\cup\lbrace\infty\rbrace$. Let $R$ be any $k$-algebra.
In this subsection we recall some definitions and well-known facts about HS-derivations.

\begin{definition}\label{HSdef1}
 We say that $\mathbb{D}=(D_{\mathbf{i}}:R\to R)_{\mathbf{i}\in\mathbb{N}^e}$ is an \emph{$e$-dimensional HS-derivation over $k$} if
  the map
  $$\mathbb{D}:R \to R\llbracket \bar{X}\rrbracket,\qquad r\mapsto \sum\limits_{\mathbf{i}\in\mathbb{N}^e} D_{\mathbf{i}}(r)\bar{X}^{\mathbf{i}},$$
  where $\bar{X}^{\mathbf{i}}=X_1^{i_1}\cdot\ldots\cdot X_e^{i_e}$ for $\mathbf{i}=(i_1,\ldots,i_e)$,
  is a $k$-algebra homomorphism and $D_{\mathbf{0}}=\id_R$.
\end{definition}
We introduce $R[\bar{v}]:=R[\bar{X}]/(X_1^{p^m},\ldots,X_e^{p^m})$, so $v_i=X_i+(X_1^{p^m},\ldots,X_e^{p^m})$ and $\bar{v}=(v_1,\ldots,v_e)$
(for $m=\infty$ we set $v_i=X_i$, $R[\bar{v}]=R\llbracket\bar{X}\rrbracket$).
After composing $\mathbb{D}$
with the natural mapping $R[\bar{X}]\to R[\bar{v}]$ we obtain a truncation of $\mathbb{D}$, 
denoted by $\mathbb{D}[m]=(D_{\mathbf{i}}:R\to R)_{\mathbf{i}\in [p^m]^e}$. This lead us to the following:

\begin{definition}
 A collection $\mathbb{D}=(D_{\mathbf{i}}:R\to R)_{\mathbf{i}\in [p^m]^e}$ is called an \emph{$m$-truncated $e$-dimensional HS-derivation over $k$} if
 the map 
 $$\mathbb{D}:R \to R[\bar{v}],\qquad r\mapsto \sum\limits_{\mathbf{i}\in [p^m]^e} D_{\mathbf{i}}(r)\bar{v}^{\mathbf{i}},$$
 where $\bar{v}^{\mathbf{i}}=v_1^{i_1}\cdot\ldots\cdot v_e^{i_e}$ for $\mathbf{i}=(i_1,\ldots,i_e)$,
 is a $k$-algebra homomorphism and $D_{\mathbf{0}}=\id_R$.
\end{definition}
Clearly, any $\infty$-truncated HS-derivation is just an HS-derivation.
We have seen that it is easy to obtain from an HS-derivation an $m$-truncated one. For a field $R=K$ the converse is also true.

\begin{theorem}\label{HSderiv}
 Let $R$ be a smooth $k$-algebra, $\mathbb{D}=(D_{\mathbf{i}}:R\to R)_{\mathbf{i}\in [p^m]^e}$ an $m$-truncated $e$-dimensional
 HS-derivation over $k$. There exists an $e$-dimensional HS-derivation $\mathbb{D}'=(D'_{\mathbf{i}}:R\to R)_{\mathbf{i}\in \mathbb{N}^e}$
 over $k$ such that for every $\mathbf{i}\in [p^m]^e$ we have $D_{\mathbf{i}}'=D_{\mathbf{i}}$.
\end{theorem}

\begin{proof}
 We recursively construct $\mathbb{D}'$ as was done at \cite[p. 236]{Mats1}, but using
 the following diagram
\begin{equation*}
 \xymatrixcolsep{4.5pc}\xymatrixrowsep{1.5pc}\xymatrix{ 
  R \ar[r]^(.35){\varphi} \ar@{-->}[dr] & R[\overline{X}]/(X_1^{p^n},\ldots,X_e^{p^n}) 
  \\ k \ar[u] \ar[r]   & R[\overline{X}]/(X_1^{p^{n+1}},\ldots,X_e^{p^{n+1}}) \ar[u]_{\pi} }
\end{equation*} 
where $\varphi(x):=\sum\limits_{\mathbf{i}\in[p^n]^e}D_{\mathbf{i}}(x)\bar{X}^{\mathbf{i}}+(X_1^{p^n},\ldots,X_e^{p^n})$
and $\pi$ is the quotient map.
\end{proof}

\begin{remark}
 Theorem \ref{HSderiv} is a generalisation of \cite[Thm. 6.]{Mats1}. Note that the best possible situation is for
 a $k$-algebra $R$ which is \'{e}tale over $k$. In such a case there exists a unique expansion of every $m$-truncated
 $e$-dimensional HS-derivation.
\end{remark}
\noindent
By \cite[Thm. 26.9]{mat}, separability implies smoothness, so Theorem \ref{HSderiv} works
 in particular for a separable fields extension $k\subseteq K$.
Because so far we do not demand
anything from $k$ we can take $k=\mathbb{F}_p$, hence the assumption about a separable extension $k\subseteq K$ is negligible in the following way:

\begin{cor}\label{integr}
 Every $m$-truncated $e$-dimensional HS-derivation on a field $K$ has an extension to an $e$-dimensional HS-derivation.
\end{cor}
\noindent
We call an $m$-truncated $e$-dimensional HS-derivation $\mathbb{D}$ on $R$ \textit{integrable} if there exists $e$-dimensional HS-derivation
$\mathbb{D}'$ on $R$ such that $D'_{\mathbf{i}}=D_{\mathbf{i}}$ for every $\mathbf{i}\in [p^m]^e$. Corollary \ref{integr} says
that truncated HS-derivations on a field are always integrable, but it is not true for arbitrary rings \cite[Example 3.]{Mats1}.
Moreover, the described situation dramatically changes after adding some iterativity conditions.
Before considering iterative HS-derivations, we state more well-known facts about general HS-derivations, which will be needed in the sequel.

\begin{lemma}\label{smooth}
 Assume that $R\xrightarrow{f} S$ is a homomorphism of $k$-algebras. Let $\mathbb{D}$ be an $m$-truncated
 $e$-dimensional  HS-derivation on $R$ over $k$.
 \begin{enumerate}
  \item[i)] If $S$ is smooth over $R$, then
  there exists an $m$-truncated $e$-dimensional HS-derivation $\mathbb{D}'$ on $S$ over $k$
  such that for every $i_1,\ldots,i_e<p^m$
 \begin{equation}\label{etale01}
  fD_{(i_1,\ldots,i_e)}=D'_{(i_1,\ldots,i_e)}f.
 \end{equation}
 \item[ii)] If $S$ is unramified over $R$, then
  there exists at most one $m$-truncated $e$-dimensional HS-derivation $\mathbb{D}'$ on $S$ over $k$
  such that for every $i_1,\ldots,i_e<p^m$ 
 \begin{equation*}
  fD_{(i_1,\ldots,i_e)}=D'_{(i_1,\ldots,i_e)}f.
 \end{equation*}
 \end{enumerate}
\end{lemma}

\begin{proof}
 The lemma just reformulates \cite[Prop. 3.3]{HK}. 
\end{proof}

\begin{fact}
For every $m$-truncated $e$-dimensional HS-derivation and every $x\in R$ the following holds
 $$D_{(i_1,\ldots,i_e)}(x^p)=\begin{cases} D_{(\frac{i_1}{p},\ldots,\frac{i_e}{p})}(x)^p &\mbox{if } p|i_1,\ldots,i_e
			      \\ 0 &\mbox{otherwise.}
                 \end{cases}$$
\end{fact}

\begin{proof}
 It follows from the definition (see e.g. \cite[Lemma 1.1]{okugawa}).	
\end{proof}

\subsection{Iterative HS-derivations}
In this subsection we deal with iterative HS-derivations. The main purpose is to provide basic properties. 
Let $F(\bar{v},\bar{w})=(F_1(\bar{v},\bar{w}),\ldots,F_e(\bar{v},\bar{w}))\in \big(k[\bar{v},\bar{w}]\big)^e$
(still $k[\bar{v},\bar{w}]=k\llbracket\bar{X},\bar{Y}\rrbracket$ for $m=\infty$) and let $\mathbb{D}$ be an $m$-truncated
$e$-dimensional HS-derivation on $R$ over $k$. Sometimes we need to distinguish between $\mathbb{D}:R\to R[\bar{v}]$
and $\mathbb{D}:R\to R[\bar{w}]$, therefore they will be denoted by $\mathbb{D}_{\bar{v}}$ and $\mathbb{D}_{\bar{w}}$ respectively.

\begin{definition}\label{defiterativ}
 We call $\mathbb{D}$ \emph{$F$-iterative} if the following diagram commutes
 \begin{equation*}
 \xymatrixcolsep{4.5pc}\xymatrixrowsep{1.5pc}\xymatrix{ 
  R \ar[r]^{\mathbb{D}_{\bar{v}}} \ar[d]_{\mathbb{D}_{\bar{v}}} & R[\bar{v}] \ar[d]^{\mathbb{D}_{\bar{w}}[\bar{v}]}
  \\ R[\bar{v}]  \ar[r]_{\ev_F}   & R[\bar{v},\bar{w}] }
\end{equation*}
where $\mathbb{D}_{\bar{w}}[\bar{v}](\sum\limits_{\mathbf{i}}r_{\mathbf{i}}\bar{v}^{\mathbf{i}}):=
\sum\limits_{\mathbf{i}}\mathbb{D}_{\bar{w}}(r_{\mathbf{i}})\bar{v}^{\mathbf{i}}$. We will write shortly \emph{$F$-derivation} for
an $F$-iterative $m$-truncated $e$-dimensional HS-derivation over $k$.
\end{definition}
\noindent

\begin{example}
 For $m=\infty$ and $e=1$ we can take $F=\mathbb{G}_a=X+Y$. It encodes the classical iterativity rule
 $$D_i\circ D_j={i+j\choose i}D_{i+j}.$$
 An example of a $\mathbb{G}_a$-derivation is the following collection of functions on $k[X]$:
 $$D_n\big(\sum\limits_{i=0}^{k}\alpha_iX^i\big)=\begin{cases} 0,\quad\text{if}\;\;\;n>k
\\ \sum\limits_{i=n}^{k}\alpha_i{i\choose n}X^{i-n},\quad\text{if}\;\;\;n\le k,
\end{cases}$$
where $n\in\mathbb{N}$.
For the formal group law $F=\mathbb{G}_m=X+Y+XY$ above formulas are more complicated (see \cite[Example 3.6]{HK1}).
\end{example}

\begin{example}\label{exDF}
 For every formal group law $F(\bar{X},\bar{Y})\in(k\llbracket\bar{X},\bar{Y}\rrbracket)^e$ we have \textit{canonical $F$-derivation}
 $$\mathbb{D}^F:=\ev_{F(\bar{X},\bar{Y})}:k\llbracket\bar{X}\rrbracket\to k\llbracket\bar{X}\rrbracket\llbracket\bar{Y}\rrbracket.$$
Compare with \cite[Example 3.25]{HK}.
 \end{example}

\begin{example}
For actions of finite group schemes, which underlying Hopf algebra is defined on $k[\bar{v}]$, we have a natural correspondence with the truncated $F$-derivations for an appropriate $F$ (see Section 3. in \cite{HK}). Therefore we are especially interested in group scheme actions of $k$-group schemes of the form $\mathfrak{g}=\spe k[\bar{v}]$ on the scheme $\spe R$. 
By \cite[Remark 3.9]{HK}, such a group scheme action corresponds to an $F$-derivation on $R$, where $F$ is the Hopf algebra comultiplication given by $\mathfrak{g}$.
\end{example}

\noindent
Assume that $R$ is a $k$-algebra with an $F$-derivation $\mathbb{D}$. The pair $(R,\mathbb{D})$ will be called an \emph{$F$-ring}.
If $K$ is a field and $(K,\mathbb{D})$ is an $F$-ring, then $(K,\mathbb{D})$ will be called an \emph{$F$-field}.
Let $(R,\mathbb{D})$ be an $F$-ring, similarly $(S,\mathbb{D}')$. A morphism of $k$-algebras $f:R\to S$
is an \emph{$F$-morphism} if for every $\mathbf{i}$,  $fD_{\mathbf{i}}=D'_{\mathbf{i}}f$. Moreover, if such $f$ is injective,
$R$ is \emph{$F$-subring} of $S$ (similarly \emph{$F$-subfield} for $F$-fields).

\begin{example}\label{exDG}
 Let $G$ be an algebraic group over $k$, we denote by $\mathcal{O}_G$ the local ring of $G$ at the identity (it is a regular local ring) and
 by $\bar{x}=(x_1,\ldots,x_e)$ a choice of its local parameters. For $F=\hat{G}$ we have $F(\bar{x},\bar{Y})\in\mathcal{O}_G\llbracket \bar{Y}\rrbracket$,
 so $(\mathcal{O}_G,\mathbb{D}^{F}|_{\mathcal{O}_G})$ is an $F$-subring of $(k\llbracket\bar{X}\rrbracket,\mathbb{D}^F)$.
 Hence $k(G)$ is equipped with a natural $\hat{G}$-derivation, which will be denoted by $\mathbb{D}^G$ and called \textit{canonical $G$-derivation}.
 It depends on the choice of local parameters, but we prefer the adjective ``canonical''.
 For more details check \cite[Example 3.27]{HK}.
\end{example}

\noindent
For an $F$-ring $(R,\mathbb{D})$ and $\mathbf{i}\in [p^m]^{e}$ we introduce $C_{\mathbf{i}}:=\ker D_{\mathbf{i}}$, and two more sets:
$$C_R:=C_{(1,0\ldots,0)}\cap\ldots\cap C_{(0,\ldots,0,1)}\quad\text{ (the ring of constants)},$$
and
$$C_R^{\text{abs}}:=\bigcap\limits_{\mathbf{i}\not=\mathbf{0}}C_{\mathbf{i}}\quad\text{ (the ring of absolute constants)}.$$
Both, $C_R$ and $C_R^{\text{abs}}$, are subrings of $R$ (see Remark \ref{Fs}).

\begin{lemma}\label{smooth2}
 Assume that $R\xrightarrow{f} S$ is a homomorphism of $k$-algebras. Let $\mathbb{D}$ be an $F$-derivation on $R$.
 \begin{enumerate}
  \item[i)] If $S$ is \'{e}tale (smooth and unramified) over $R$, then
  there exists a unique $F$-derivation $\mathbb{D}'$ on $S$
  such that for every $i_1,\ldots,i_e<p^m$
 \begin{equation*}\label{etale01}
  fD_{(i_1,\ldots,i_e)}=D'_{(i_1,\ldots,i_e)}f.
 \end{equation*}
 \item[ii)] If $S$ is unramified over $R$, then
  there exists at most one $F$-derivation $\mathbb{D}'$ on $S$
  such that for every $i_1,\ldots,i_e<p^m$ 
 \begin{equation*}
  fD_{(i_1,\ldots,i_e)}=D'_{(i_1,\ldots,i_e)}f.
 \end{equation*}
 \end{enumerate}
\end{lemma}

\begin{proof}
 Compare to \cite[Prop. 3.18]{HK}. 
 Part ii) is, by Lemma \ref{smooth}.ii), true even without the iterativity assumption.
 For the proof of part i), it is enough to show that an HS-derivation $\mathbb{D}'$ from Lemma \ref{smooth} is $F$-iterative,
 i.e. the following diagram is commutative
  \begin{equation*}
  \xymatrixcolsep{4pc}\xymatrix{ S \ar[r]^{\mathbb{D}'_{\bar{v}}} \ar[d]_{\mathbb{D}'_{\bar{v}}} & S[\overline{v}] \ar[d]^{\mathbb{D}'_{\bar{w}}[\overline{v}]} \\
	   S[\overline{v}] \ar[r]_{\ev_{F}} & S[\overline{v}, \overline{w}]}
 \end{equation*}
It is similar to the proof of \cite[Theorem 27.2]{mat} and we leave it to the reader.
\end{proof}

\noindent
Let $F(\bar{v},\bar{w})\in(k[\bar{v},\bar{w}])^e$,
$m'\le m$ and let $\bar{v'}$, $\bar{w'}$ denote the $m'$-truncated variables ($k[\bar{v'},\bar{w'}]=k[\bar{X},\bar{Y}]/(X_1^{p^{m'}},\ldots,X_e^{p^{m'}},
Y_1^{p^{m'}},\ldots,Y_e^{p^{m'}})$).
By $F[m']$ we denote the $m'$-truncation of $F$ which is equal to $\ev_{(\bar{v'},\bar{w'})}F(\bar{v},\bar{w})$
(the image of $F$ in the ring of truncated polynomials $\big(k[\bar{v'},\bar{w'}]\big)^e$).
If $\mathbb{D}$ is $F$-iterative,
then $\mathbb{D}[m']$ is $F[m']$-iterative as well (for the notion of $\mathbb{D}[m]$, check the first lines after Definition \ref{HSdef1}).

\begin{example}
  For every $m\in\mathbb{N}_{>0}$ we get a $\hat{G}[m]$-field structure on $k(G)$ - just consider $\mathbb{D}^G[m]$. 
\end{example}

\begin{definition}\label{defintegrable}
 Let $F(\bar{X},\bar{Y})\in(k\llbracket\bar{X},\bar{Y}\rrbracket)^e$ and let $\mathbb{D}$ be an $F[m]$-derivation on a $k$-algebra $R$.
 We call $\mathbb{D}$ \emph{strongly integrable} if there exists an $F$-derivation $\mathbb{D}'$ on $R$ such that
 $\mathbb{D}'[m]=\mathbb{D}$.
\end{definition}

\noindent
In the next few facts we give simple properties of $F$-derivations on a $k$-algebra $R$.
Those facts were intended for a formal group law $F$, 
but it is enough to demand that
$F(\bar{v},\bar{w})\in(k[\bar{v},\bar{w}])^e$, $F(\bar{v},\bar{0})=\bar{v}$ and $F(\bar{0},\bar{w})=\bar{w}$. However, we do not consider
$F$-derivations in the case when $F$ is not a formal group law, even the existence for such (non-trivial) derivations is not clear in general.

\begin{fact}\label{Newton-comp} For every $\mathbf{i}$ and $\mathbf{j}$ there exists
$\mathbf{r}(D_{\mathbf{j}'})_{0<|\mathbf{j}'|<|\mathbf{i}+\mathbf{j}|}$, a $k$-linear
combination of $D_{\mathbf{j}'}$, where $0<|\mathbf{j}'|<|\mathbf{i}+\mathbf{j}|$, such that
 $$D_{\mathbf{j}}D_{\mathbf{i}}=
 {i_1+j_1 \choose i_1}\ldots{i_e+j_e \choose i_e}D_{\mathbf{i}+\mathbf{j}}
 +\mathbf{r}(D_{\mathbf{j}'})_{0<|\mathbf{j}'|<|\mathbf{i}+\mathbf{j}|}.$$
\end{fact}

\begin{proof}
 It is clear for $\mathbf{i}=\mathbf{0}$ or $\mathbf{j}=\mathbf{0}$, so assume that both $\mathbf{i}$ and $\mathbf{j}$ differ from $\mathbf{0}$.
 Since $F(\bar{v},0)=\bar{v}$, $F(0,\bar{w})=\bar{w}$, we have $F(\bar{v},\bar{w})=(v_1+w_1+S_1,\ldots,v_e+w_e+S_e)$
 for some $S_1,\ldots,S_e$ belonging to the ideal $(v_iw_j)_{i,j\le e}$. Therefore for every $r\in R$
 $$\sum\limits_{j_1,\ldots,j_e,i_1,\ldots,i_e} D_{(j_1,\ldots,j_e)}D_{(i_1,\ldots,i_e)}(r)v_1^{i_1}\cdot\ldots\cdot v_e^{i_e}\cdot w_1^{j_1}\cdot\ldots\cdot w_e^{j_e}=$$
 $$\sum\limits_{k_1,\ldots,k_e} D_{(k_1,\ldots,k_e)}(r)(v_1+w_1+S_1)^{k_1}\cdot\ldots\cdot (v_e+w_e+S_e)^{k_e}.$$
We are interested in the coefficients at $A:=v_1^{i_1}\cdot\ldots\cdot v_e^{i_e}\cdot w_1^{j_1}\cdot\ldots\cdot w_e^{j_e}$ on the right side
of the above equation.
First of all, note that $v_i+w_i+S_i$, $i\leqslant e$, is an element of the maximal ideal $(v_1,\ldots,v_e,w_1,\ldots,w_e)$, hence it is of the form
$$\alpha_1v_1+\ldots+\alpha_ev_e+\beta_1 w_1+\ldots+\beta_ew_e,$$
for some $\alpha_1,\ldots,\alpha_e,\beta_1,\ldots,\beta_e\in k[\bar{v},\bar{w}]$.
Each component of the above sum has total degree at least $1$, so the total degree of each summand of $(v_i+w_i+S_i)^{k_i}$ is at least $k_i$. Therefore the total degree of
$$(v_1+w_1+S_1)^{k_1}\cdot\ldots\cdot (v_e+w_e+S_e)^{k_e}$$
it at least equal to $k_1+\ldots+k_e$. On the other hand, the total degree of $A$
is equal to $|\mathbf{i}+\mathbf{j}|$.
After comparing degrees, we see that if $k_1+\ldots+k_e>|\mathbf{i}+\mathbf{j}|$ then there is no chance to find a component of
$$D_{(k_1,\ldots,k_e)}(r)(v_1+w_1+S_1)^{k_1}\cdot\ldots\cdot (v_e+w_e+S_e)^{k_e}$$
equal to $A$ multiplied by some element of $R$.

Let $k_1+\ldots+k_e=|\mathbf{i}+\mathbf{j}|$. Since $S_1,\ldots,S_e\in (v_iw_j)_{i,j\le e}$, each summand of $S_i$, $i\leqslant e$, has total degree at least $2$. The only component of
$$D_{(k_1,\ldots,k_e)}(r)(v_1+w_1+S_1)^{k_1}\cdot\ldots\cdot (v_e+w_e+S_e)^{k_e}$$
for which the total degree will be equal $|\mathbf{i}+\mathbf{j}|$
``omits'' $S_1,\ldots,S_e$. Therefore we are looking for the coefficient of
$$D_{(k_1,\ldots,k_e)}(r)(v_1+w_1)^{k_1}\cdot\ldots\cdot (v_e+w_e)^{k_e},$$
which is divisible by $A$.
\end{proof}

\begin{fact}\label{Dpp-gen}
 Assume that also $\mathbb{D}'$ is an $F$-derivation on $R$.
 If for all $l\le e$ and $i<m$ we have
 $D_{(0,\ldots,0,\underset{l\text{-th place}}{p^i},0,\ldots,0)}=D'_{(0,\ldots,0,\underset{l\text{-th place}}{p^i},0,\ldots,0)}$ then
 $\mathbb{D}=\mathbb{D}'$.
\end{fact}

\begin{proof}
 Induction on $|\mathbf{j}|$. 
 Clearly, $D_{(0,\ldots,0)}=\id_R=D'_{(0,\ldots,0)}$.
 Take $\mathbf{j}=(j_1,\ldots,j_e)\neq (0,\ldots,0)$
 and assume that $D_{\mathbf{j}'}=D'_{\mathbf{j}'}$
 for every $\mathbf{j}'$ such that $|\mathbf{j}'|<|\mathbf{j}|$.
 Without loss of generality, we set $j_1\neq 0$.
 Let $j_1=\gamma_0+\gamma_1p+\ldots+\gamma_sp^s$, where $\gamma_0,\ldots,\gamma_s<p$ and $\gamma_s\neq 0$.
 Fact \ref{Newton-comp} implies that
 $$D_{(p^s,0,\ldots,0)} D_{(j_1-p^s,j_2,\ldots,j_e)}=\gamma_s D_{\mathbf{j}}
 +\mathbf{r}(D_{\mathbf{j}'})_{0<|\mathbf{j}'|<|\mathbf{j}|},$$
 $$D'_{(p^s,0,\ldots,0)} D'_{(j_1-p^s,j_2,\ldots,j_e)}=\gamma_s D'_{\mathbf{j}}
 +\mathbf{r}(D'_{\mathbf{j}'})_{0<|\mathbf{j}'|<|\mathbf{j}|}.$$
 A $k$-linear combination $\mathbf{r}(D_{\mathbf{j}'})_{0<|\mathbf{j}'|<|\mathbf{j}|}$   is unique for $F$ (what can be deduced from the proof of Fact \ref{Newton-comp}), hence, by the inductive assumption, it is equal to $\mathbf{r}(D'_{\mathbf{j}'})_{0<|\mathbf{j}'|<|\mathbf{j}|}$.
Moreover, it follows from the inductive assumption that
$$D_{(p^s,0,\ldots,0)} D_{(j_1-p^s,j_2,\ldots,j_e)}=D'_{(p^s,0,\ldots,0)} D'_{(j_1-p^s,j_2,\ldots,j_e)},$$
 so $D_{\mathbf{j}}=D'_{\mathbf{j}}$.
\end{proof}

\begin{lemma}\label{wronskian}
 Let $(K,\mathbb{D})$ be an $F$-field and let
 $\partial_1,\ldots,\partial_{p^e}$ be all different elements of $\lbrace D_{(i_0,\ldots,i_e)}\;\; |\;\; i_0,\ldots,i_e<p\rbrace$.
 Take any $x_1,\ldots,x_n\in K$. Elements $x_1,\ldots,x_n$ are linearly dependent over $C_K$ if
 and only if the rank of the matrix  $\big(\partial_i(x_j)\big)_{i\le p^e,j\le n}$ is smaller than $n$.
\end{lemma}

\begin{proof}
 The proof of \cite[Proposition 3.20]{HK} works well for the above, more general lemma.
\end{proof}

\begin{cor}\label{lin_disjoint}
 For every $F$-field extension $K\subseteq L$, $K$ and $C_L$ are linearly disjoint over $C_K$.
\end{cor}

\begin{definition}
 We call an $F$-field $(K,\mathbb{D})$ \emph{strict} if $C_K=K^p$.
\end{definition}

\begin{remark}\label{strict_etale}
 Let $K\subseteq L$ be an $F$-field extension. If $K$ is strict, then $K\subseteq L$ is separable.
\end{remark}

\begin{proof}
 By Corollary \ref{lin_disjoint}, $K$ and $L^p\subseteq C_L$ are linearly disjoint over $K^p=C_K$, so by \cite[Theorem 26.4]{mat}
 $L$ is separable over $K$.
\end{proof}

\begin{lemma}\label{dimconstants}
For any $F$-field $(K,\mathbb{D})$ we have $[K:C_K]\leqslant p^e$.
\end{lemma}

\begin{proof}
 It follows from Lemma \ref{wronskian}.
\end{proof}

\subsection{Commutative HS-derivations}
In this subsection we deal with formulas for $D_{\mathbf{i}}^{(p)}$ (the $p$-th composition) in the case of
an $F$-derivation $\mathbb{D}=(D_{\mathbf{i}})_{\mathbf{i}\in[p^m]^e}$ for
a commutative $F$ (i.e. $F(\bar{v},\bar{w})=F(\bar{w},\bar{v})$).
The main idea follows Section 3.3. in \cite{HK1}, but improves the reasoning of
\cite[Proposition 3.11]{HK1} and
\cite[Remark 3.12.(4)]{HK1}.
The idea to focus on the ring of symmetric polynomials comes from Piotr Kowalski.
We assume only that $F(\bar{v},\bar{w})\in(k[\bar{v},\bar{w}])^e$ is commutative and that $(R,\mathbb{D})$ is an $F$-ring.
Obviously:

 \begin{fact}\label{commut1}
 We have the following $$D_{\mathbf{j}}\circ D_{\mathbf{i}} = D_{\mathbf{i}}\circ D_{\mathbf{j}}.$$
\end{fact}

For every $N\ge 1$ we introduce the following $k$-algebra homomorphism
$$E_N: R[\bar{v}_1,\ldots,\bar{v}_{N-1}]\to R[\bar{v}_1,\ldots,\bar{v}_N],$$
$$E_N=\mathbb{D}_{\bar{v}_N}[\bar{v}_1,\ldots,\bar{v}_{N-1}],$$
where $\bar{v}_1,\ldots,\bar{v}_N$ are $e$-tuples of $m$-truncated variables and
$$\mathbb{D}_{\bar{v}_N}[\bar{v}_1,\ldots,\bar{v}_{N-1}]\Big(\sum\limits_{\mathbf{i}_1,\ldots,\mathbf{i}_{N-1}}\alpha_{\mathbf{i}_1,\ldots,\mathbf{i}_{N-1}}
\bar{v}_1^{\mathbf{i}_1}\cdot\ldots\cdot \bar{v}_{N-1}^{\mathbf{i}_{N-1}}\Big)=$$
$$=\sum\limits_{\mathbf{i}_1,\ldots,\mathbf{i}_{N}}D_{\mathbf{i}_N}(\alpha_{\mathbf{i}_1,\ldots,\mathbf{i}_{N-1}})
\bar{v}_1^{\mathbf{i}_1}\cdot\ldots\cdot \bar{v}_{N}^{\mathbf{i}_{N}}.$$
 For $N\ge 1$ we define inductively
 $$F_1(\bar{v}_1):=\bar{v}_1,$$
 $$F_{N+1}(\bar{v}_1,\ldots,\bar{v}_{N+1}):=F_N\big(\bar{v}_1,\ldots,\bar{v}_{N-1},F(\bar{v}_N,\bar{v}_{N+1})\big).$$

 \begin{lemma}\label{wzor00}
 For every $N\ge 1$ the following diagram commutes
 \begin{equation*}
   \xymatrixcolsep{2.5pc}\xymatrixrowsep{1.5pc}\xymatrix{R \ar[rr]^{E_N\circ\ldots\circ E_1} \ar[dr]_{E_1}& & R[\bar{v}_1,\ldots,\bar{v}_{N}]
   \\ & R[\bar{v}_1] \ar[ur]_{\ev_{F_N}}&}
 \end{equation*}
\end{lemma}
 
\begin{proof} 
It is clear for $N=1$, so assume for the induction step that the last diagram is commutative. Consider
 \begin{equation*}
   \xymatrixcolsep{6.5pc}\xymatrixrowsep{1.5pc}\xymatrix{R \ar[r]^{E_{N-1}\circ\ldots\circ E_1} \ar[d]_{E_1}  & 
   R[\bar{v}_1,\ldots,\bar{v}_{N-1}] \ar[d]_{E_N} \ar[r]^{E_{N}} &R[\bar{v}_1,\ldots,\bar{v}_{N}] \ar[d]^{E_{N+1}}\\
   R[\bar{v}_1] \ar[r]_{\ev_{F_N}} & R[\bar{v}_1,\ldots,\bar{v}_{N}] \ar[r]_{\ev_{(\bar{v}_1,\ldots,\bar{v}_{N-1},F(\bar{v}_N,\bar{v}_{N+1})}}
   & R[\bar{v}_1,\ldots,\bar{v}_{N+1}]}
 \end{equation*}
Left part is commutative by the inductive assumption. For commutativity of the right side, just apply the functor
$$R\to R[\bar{v}_1,\ldots,\bar{v}_{N-1}]$$
to the diagram from the $F$-iterativity definition and change $\bar{v}$, $\bar{w}$ to $\bar{v}_{N}$, $\bar{v}_{N+1}$.
Finally
$$\ev_{\big(\bar{v}_1,\ldots,\bar{v}_{N-1},F(\bar{v}_N,\bar{v}_{N+1})\big)}\circ\ev_{F_N}=\ev_{F_{N+1}}.$$
\end{proof}
Note that the following composition of mappings 
$$R\xrightarrow{E_1} R[\bar{v}_1]\xrightarrow{E_2} R[\bar{v}_1,\bar{v}_2]\to\ldots\xrightarrow{E_p} R[\bar{v}_1,\ldots,\bar{v}_p]$$
%\xrightarrow{\ev_{(\bar{v},\ldots,\bar{v})}} R[\bar{v}]$$
is a $k$-algebra homomorphism such that $\im (E_p\circ\ldots\circ E_1)$ is,
by Fact \ref{commut1}, a subset of the ring of symmetric polynomials in $\bar{v}_1,\ldots,\bar{v}_p$,
i.e.: elements of $R[\bar{v}_1,\ldots,\bar{v}_p]$ invariant under the action of $S_p$,
$$\sigma: \bar{v}_i=(v_{i,1},\ldots,v_{i,e})\mapsto \bar{v}_{\sigma(i)}=(v_{\sigma(i),1},\ldots,v_{\sigma(i),e}),$$
for $\sigma\in S_p$ and $i\le p$. 
In other words the map $E_p\circ\ldots\circ E_1$ factors as in following the diagram
\begin{equation*}
   \xymatrixcolsep{6.5pc}\xymatrixrowsep{1.5pc}\xymatrix
   {R \ar[dr]_{E_{p}\circ\ldots\circ E_1} \ar@{-->}[r]  & R[\bar{v}_1,\ldots,\bar{v}_{p}]^{S_p} \ar[d]^{\subseteq} \\
    & R[\bar{v}_1,\ldots,\bar{v}_{p}]}
\end{equation*}
For $\varphi:  R[\bar{v}_1,\ldots,\bar{v}_{p}]^{S_p} \to R[\bar{v}_1^{\frac{1}{p}}]$, given by
by $v_{i,j}\mapsto v_{1,j}^{\frac{1}{p}}$, where $i\le p$ and $j\le e$, 
also the map $\varphi$ factors as in the following diagram
\begin{equation*}
   \xymatrixcolsep{6.5pc}\xymatrixrowsep{1.5pc}\xymatrix
   {R[\bar{v}_1,\ldots,\bar{v}_{p}]^{S_p}  \ar[dr]_{\varphi}   \ar@{-->}[r]  & R[\bar{v}_1]  \ar[d]^{\subseteq} \\
    & R[\bar{v}_1^{\frac{1}{p}}]}
\end{equation*}
Therefore $\varphi: \im (E_p\circ\ldots\circ E_1) \to R[\bar{v}_1]$, 
defined 
is a well-defined $k$-algebra homomorphism.

For any $N\ge 1$ we define inductively the ``multiplication by $N$ map'':
$$[1]_F:=\bar{v},$$
$$[N+1]_F:=F(\bar{v},[N]_F).$$
For example $[2]_F=F(\bar{v},\bar{v})$.

\begin{cor}\label{evP}
 For any $r\in R$ we have
 $$\sum\limits_{\mathbf{i}}D_{\mathbf{i}}^{(p)}(r)\bar{v}^{\mathbf{i}}=
 \ev_{[p]_F(\bar{v}^{1/p})}\big(\sum\limits_{\mathbf{i}}D_{\mathbf{i}}(r)\bar{v}^{\mathbf{i}}\big) .$$
\end{cor}

\begin{proof}
 By Lemma \ref{wzor00} we know that
 $$E_p\circ\ldots\circ E_1(r)=\ev_{F_p}\circ E_1(r),$$
 so
 $$\sum\limits_{\mathbf{i}}D_{\mathbf{i}}^{(p)}(r)\bar{v}_1^{\mathbf{i}}=\varphi\circ E_p\circ\ldots\circ E_1(r)=
 \varphi\circ \ev_{F_p}\circ E_1(r)=
 \ev_{[p]_F(\bar{v}_1^{1/p})}\big(\sum\limits_{\mathbf{i}}D_{\mathbf{i}}(r)\bar{v}_1^{\mathbf{i}}\big).$$
 The first equality is similar to \cite[Lemma 3.7]{HK1}, the last follows from definitions of $[p]_F$, $F_p$ and $\varphi$.
 For example let $p=2$:
 $$\varphi\circ\ev_{F_2(\bar{v}_1,\bar{v}_2)}=\ev_{F_2(\bar{v}_1^{1/p},\bar{v}_1^{1/p})}=\ev_{[2]_F(\bar{v}_1^{1/p})}.$$
\end{proof}

\section{Canonical $G$-bases and the Integrability}\label{secsetup}
The results of this section focus on proving the integrability for a field equipped with an iterative HS-derivation and endowed
with a $p$-basis of a special kind.
For the notion of $p$-independence, $p$-basis and their basic properties, the reader is referred to \cite[p. 202.]{mat}.
Recall that $k$ is a perfect field. Assume that $G$ is an algebraic group over $k$ of dimension $e$ (perhaps not commutative). We will write
$G[m]$-derivation, $G[m]$-ring, $G[m]$-field, $\ldots$ instead of $\hat{G}[m]$-derivation, $\hat{G}[m]$-ring, $\hat{G}[m]$-field, $\ldots$
\
\\
\\
Let $(K,\mathbb{D})$ be a $G[m]$-field.
For every $s\in\lbrace 0,\ldots,m-1\rbrace$ we introduce
\begin{equation*}
 F_s:=\bigcap\limits_{j=0}^{s}C_{(p^j,0,\ldots,0)}\cap C_{(0,p^j,0,\ldots,0)}\cap\ldots\cap C_{(0,\ldots,0,p^j)},\quad F_{-1}:=K.
\end{equation*}
\begin{remark}\label{Fs}
Sets $F_s$ are, due to Fact \ref{Dpp-gen}, subfields of $K$. In fact $F_s$ is equal to 
the field of
constants of order $s$ (the absolute constants of $\mathbb{D}[s+1]$).
\end{remark}

For the clarity of the following proofs, we note an obvious fact:
\begin{fact}\label{kombajn}
 Let $L\subseteq L'$ be an extension of fields.
 If $y\in L^{\frac{1}{p}}\backslash L\subseteq L'$,
 then $[L(y):L]=p$.
\end{fact}

\begin{lemma}\label{claim1}
 Let $z_1,\ldots,z_e\in K$ form a $p$-basis of (or equivalently, by Lemma \ref{dimconstants}, ``are $p$-independent in'') $K$ over $C_K=F_0$.
 For every $s\in\lbrace 0,\ldots,m-1\rbrace$ we have
 $$[F_{s-1}:F_s]=p^e,\qquad F_{s-1}=F_s(z_1^{p^s},\ldots,z_e^{p^s}).$$
\end{lemma}

\begin{proof}
  Being a $p$-basis for $K$ over $F_0$, due to $K^p\subseteq F_0$, means that $K=F_0(z_1,\ldots,z_e)$ and that $[F_0(z_1,\ldots,z_e):F_0]=p^e$.
Notice that, by Lemma 3.31 from \cite{HK} and Lemma \ref{dimconstants},
$$[F_s(z_1^{p^s},\ldots,z_e^{p^s}):F_s]\le [F_{s-1}:F_s]\le p^e.$$ 
It is enough to show that $[F_s(z_1^{p^s},\ldots,z_e^{p^s}):F_s]=p^e$. 
We know that $\lbrace z_1^{i_1}\cdot\ldots\cdot z_e^{i_e}\;\;|\;\; 0\le i_1,\ldots, i_e<p\rbrace$ is $F_0$-linearly independent, thus
$\lbrace z_1^{i_1p^s}\cdot\ldots\cdot z_e^{i_ep^s}\;\;|\;\; 0\le i_1,\ldots, i_e<p\rbrace$ is $F_0^{p^s}$-linearly independent.
Consider
$$(K^{p^s}, \mathbb{D}[s+1]|_{K^{p^s}})\subseteq (F_{s-1},\mathbb{D}[s+1]|_{F_{s-1}}).$$
By \cite[Lemma 3.31]{HK}, it is an extension of $G[1]^{\fr^{p^s}}$-fields. Therefore, by Corollary \ref{lin_disjoint},
$K^{p^s}$ is linearly disjoint from constants of $\mathbb{D}[s+1]|_{F_{s-1}}$ over constants of $\mathbb{D}[s+1]|_{K^{p^s}}$.
So $F_0^{p^s}$-linear independence of $\lbrace z_1^{i_1p^s}\cdot\ldots\cdot z_e^{i_ep^s}\;\;|\;\; 0\le i_1,\ldots, i_e<p\rbrace$
implies it $F_s$-linear independence. Hence $[F_s(z_1^{p^s},\ldots,z_e^{p^s}):F_s]=p^e$.
\end{proof}

\begin{remark}\label{remclaim1}
 The equality
 $$[F_{s-1}:F_s]=p^e,$$
 where $s\in\lbrace 0,\ldots,m-1\rbrace$, does not depend on the choice of a $p$-basis. Therefore it is true if $[K:C_K]=p^e$.
\end{remark}

\begin{prop}\label{lemat.27.3.ii}
 Let $z_1,\ldots,z_e\in K$ form a $p$-basis of (or equivalently ``are $p$-independent in'') $K$ over $C_K=F_0$.
 Then there exists a subset $\mathcal{B}_{0}\subseteq C_K^{\text{abs}}=F_{m-1}$,
 for which $\mathcal{B}:=\mathcal{B}_{0}\cup\lbrace z_1,\ldots,z_e\rbrace$
 is a $p$-basis for $K$ over $k$.
\end{prop}

\begin{proof}
In particular, Lemma \ref{claim1} implies that the set $\lbrace z^{p^m}_{1},\ldots
z^{p^m}_{e}\rbrace$ is $p$-independent over $k$ in $F_{m-1}$.
Let $\mathcal{B}'$ be
a $p$-basis $\mathcal{B}^{\prime}$ of the field $F_{m-1}$ over
$k$ of the form $\mathcal{B}'=\mathcal{B}_0\cupdot\lbrace z_1^{p^m},\ldots,z_e^{p^m}\rbrace$.
We will show that $\mathcal{B}:=\mathcal{B}_0\cup\lbrace z_1,\ldots,z_e\rbrace$ is a $p$-basis of $K$ over $k$.
Since for $s=m$, $\mathcal{B}_0\cup\lbrace z^{p^m}_1,\ldots,z^{p^m}_e\rbrace=\mathcal{B}^{\prime}$ is,
as above, a $p$-basis for $F_{m-1}$ over $k$, it is enough to show the following induction step
\begin{center}
 if $\mathcal{B}_0\cup\lbrace z^{p^s}_1,\ldots,z^{p^s}_e\rbrace$ is $p$-basis for $F_{s-1}$ over $k$, 
\\ then
 $\mathcal{B}_0\cup\lbrace z^{p^{s-1}}_1,\ldots,z^{p^{s-1}}_e\rbrace$ is $p$-basis for $F_{s-2}$ over $k$,
\end{center}
where $s$ descends from $m$ to $1$. 
\
\\
Firstly, we argue for the $p$-independence of $\mathcal{B}_0$ in $F_{s-2}$.
For any $n\in\mathbb{N}$ and pairwise distinct $x_1,\ldots,x_n\in\mathcal{B}_0$, by Lemma \ref{claim1}, we have
\begin{equation*}
 [F_{s-2}^p (x_1,\ldots,x_n):F_{s-1}^p ]=
 [F_{s-1}^p (z^{p^s}_1,\ldots,z^{p^s}_e,x_1,\ldots,x_n):F_{s-1}^p ]=p^{n+e},
\end{equation*}
using Lemma \ref{claim1} again, we get
\begin{IEEEeqnarray}{rCl}
p^{n+e} &=& [F_{s-2}^p (x_1,\ldots,x_n):F_{s-2}^p ]\cdot[F_{s-1}^p (z^{p^s}_1,\ldots,z^{p^s}_e):F_{s-1}^p ]\nonumber \\
&=& [F_{s-2}^p (x_1,\ldots,x_n):F_{s-2}^p ]\cdot p^e.\nonumber
\end{IEEEeqnarray}
Elements $x_1,\ldots,x_n$ were choosen arbitrary, so indeed $\mathcal{B}_0$ is $p$-independent over $ k$
in $F_{s-2}$. 
\newline
We will show now the $p$-independence of $\mathcal{B}_0\cup\lbrace z_1^{p^{s-1}},\ldots,z_e^{p^{s-1}}\rbrace$ in $F_{s-2}$.
For any $n\in\mathbb{N}$, pairwise distinct $x_1,\ldots,x_n\in\mathcal{B}_0$,
\begin{equation*}
 [F_{s-2}^p (z^{p^{s-1}}_1,\ldots,z^{p^{s-1}}_e,x_1,\ldots,x_n):F_{s-2}^p ]=
\end{equation*}
\begin{equation*}
=[F_{s-2}^p (z^{p^{s-1}}_1,\ldots,z^{p^{s-1}}_e,x_1,\ldots,x_n):
F_{s-2}^p (x_1,\ldots,x_n)]\cdot
[F_{s-2}^p (x_1,\ldots,x_n):F_{s-2}^p ].
\end{equation*}
To show the $p$-independence of $\mathcal{B}_0\cup\lbrace z^{p^{s-1}}_1,\ldots,
z^{p^{s-1}}_e\rbrace$ over $ k$ in $F_{s-2}$
we need only to prove
\begin{equation*}
 [F_{s-2}^p (z^{p^{s-1}}_1,\ldots,z^{p^{s-1}}_e,x_1,\ldots,x_n):
F_{s-2}^p (x_1,\ldots,x_n)]=p^e.
\end{equation*}
By Fact \ref{kombajn} it reduces to show that for each $i\le e$ we have
$$z_i^{p^{s-1}}\not\in F_{s-2}^p (z^{p^{s-1}}_{i+1},\ldots,z^{p^{s-1}}_e,x_1,\ldots,x_n)$$
(clearly $z^{p^s}_i\in F_{s-2}^p$).
It holds due to Lemma \ref{claim1} and $F_{s-2}^p (z^{p^{s-1}}_{i+1},\ldots,z^{p^{s-1}}_e,x_1,\ldots,x_n)
\subseteq F_{s-1}(z_{i+1}^{p^{s-1}},\ldots,z_e^{p^{s-1}})$.
\newline
Finally, we see that
\begin{IEEEeqnarray*}{rCl}
 F_{s-2} &=& F_{s-1}(z^{p^{s-1}}_1,\ldots,z^{p^{s-1}}_e) \\
         &=& F_{s-1}^p (\mathcal{B}_0\cup\lbrace z^{p^s}_1,\ldots,z^{p^s}_e\rbrace)(z^{p^{s-1}}_1,\ldots,z^{p^{s-1}}) \\
         &=& F_{s-1}^p(z^{p^s}_1,\ldots,z^{p^s}_e)(\mathcal{B}_0\cup\lbrace z^{p^{s-1}}_1,\ldots,z^{p^{s-1}} \rbrace) \\
         &=& F_{s-2}^p(\mathcal{B}_0\cup\lbrace z^{p^{s-1}}_1,\ldots,z^{p^{s-1}} \rbrace),
\end{IEEEeqnarray*}

and that ends proof of the induction, after last step we obtain that
$\mathcal{B}_0\cup\lbrace z_1,\ldots z_e\rbrace$ is a $p$-basis for $F_{1-2}=K$ over $ k$.
\end{proof}

\noindent
In the spirit of \cite[definition 6.1]{HK} we introduce the following term:

\begin{definition}\label{defbasis}
 Let $(K,\mathbb{D}[m])$ be a $G[m]$-field. A subset $B\subseteq K$ is called a \emph{canonical $G$-basis} if
 \begin{itemize}
  \item $|B|=e$,
  \item $B$ is $p$-independent in $K$ over $C_K$,
  \item there is an embedding of $G[m]$-fields $(k(G),\mathbb{D}^G[m])\to (K,\mathbb{D}[m])$ (see Example \ref{exDG})
 such that $B$ is the image of the set of canonical parameters of $G$ corresponding to the canonical $G$-derivation.
 \end{itemize}
\end{definition}

\begin{example}\label{ZieglerEx}
 Let us take $G=\mathbb{G}_a^e$. By \cite[Theorem 27.3]{mat} and Lemma \ref{product} if $[K:C_K]=p^e$ then $(K,\mathbb{D}[m])$
 has a canonical $\mathbb{G}_a^e$-basis. This fact was used in \cite{Zieg3}, to obtain the quantifier elimination for the
 theory of separably closed strict $\mathbb{G}_a^e$-fields, satisfying $[K:C_K]=p^e$.
\end{example}

\begin{theorem}\label{integrable}
 Assume that a $G[m]$-field $(K,\mathbb{D}[m])$ has a canonical $G$-basis, then $\mathbb{D}[m]$ is strongly integrable.
\end{theorem}

\begin{proof}
Let $B=\lbrace z_1,\ldots,z_e\rbrace$ be a canonical $G$-basis of $(K,\mathbb{D}[m])$ and
let $\bar{X}$ be an $e$-tuple of variables. By a choice of local parameters of $G$ at the identity we get an embedding $k(\bar{X})\subseteq k(G)$.
Proposition \ref{lemat.27.3.ii} assures the existence of a set $B_0\subseteq C_K^{\text{abs}}$
such that $B_0\cup B$ is a $p$-basis of $K$ over $k$. Let $K':=k(B_0)$.
Because $B_0\cup B$ is algebraically independent over $k$, $B_0$ is algebraically independent over $k(B)$. Moreover $k(B)\cong k(\bar{X})\subset k(G)$
is an algebraic extension, thus $B_0$ is algebraically independent over $k(G)$. Therefore $K'$ and $k(G)$ are linearly disjoint over $k$,
so the ``multiplication'' map $\mu: k(G)\otimes_k K' \to K$ is an injection, and therefore it extends to
$\tilde{\mu}:(k(G)\otimes_k K')_0 \to K$.

By \cite[Theorem 26.8]{mat}, $K'(B)$ is purely transcendental over
$K'$ and $K'(B)\subseteq K$ is $0$-\'{e}tale. 
Hence we have $K'(B)\cong K'(\bar{X})\cong (k(\bar{X})\otimes_k K')_0\subseteq (k(G)\otimes_k K')_0$
($k(G)\otimes_k K'$ is a domain as a subring of $K$). Therefore we have a natural mapping $K'(B)\to(k(G)\otimes_k K')_0=:K'(G)$.

Note that the following diagram commutes
\begin{equation*}
 \xymatrixcolsep{2.5pc}\xymatrixrowsep{1.5pc}\xymatrix{& K'(G) \ar[dr]^{\tilde{\mu}}& 
\\ K'(B)  \ar[ur] \ar[rr]_{\subseteq}
 &  & K}
\end{equation*} 
The extension of fields $K'(B)\subseteq K$ is smooth, and by \cite[Theorem 26.9]{mat} it is also separable. In particular, $K'(G)$
is separable over $K'(B)$. The algebraicity of the extension $k(B)\subseteq k(G)$ implies the algebraicity of the extension
$K'(B)\subseteq K'(G)$, and that,
due to \cite[Theorem 26.1]{mat}, means that $K'(B)\subseteq K'(G)$ is $0$-\'{e}tale. Therefore also 
$\tilde{\mu}: K'(G) \to K$ is $0$-\'{e}tale.
We have the following tower of $k$-algebras
$$k(G)\otimes_k K' \subseteq (k(G)\otimes_k K')_0 =K'(G)\xrightarrow{\tilde{\mu}} K,$$
where both extensions are $0$-\'{e}tale. By \cite[Theorem 26.7]{mat} $k(G)\otimes_k K' \xrightarrow{\mu} K$ is $0$-\'{e}tale.
\
\\
Now we are going to define a $G$-ring structure on $R:=k(G)\otimes_k K'$. For every $\mathbf{i}\in\mathbb{N}^e$, $v\in k(G)$ and $w\in K'$ 
we define
$$D_{\mathbf{i}}'(v\otimes w):=D_{\mathbf{i}}^G(v)\otimes w.$$
Note that for every $\mathbf{i}\in [p^m]^e$, we have $\mu\circ D_{\mathbf{i}}'=D_{\mathbf{i}}\circ\mu$.
Thus $(R,\mathbb{D}'[m])\xrightarrow{\mu} (K,\mathbb{D})$
is a $G[m]$-morphism. By Lemma \ref{smooth2} for $(R,\mathbb{D})$, there exists a unique $G$-derivation $\tilde{\mathbb{D}}$ on $K$.
Since both $\mathbb{D}$ and $\tilde{\mathbb{D}}[m]$ extend $\mathbb{D}'[m]$, by Lemma \ref{smooth2} for every $\mathbf{i}\in[p^m]^e$ 
we have $D_{\mathbf{i}}=\tilde{D}_{\mathbf{i}}$.
\end{proof}

\begin{remark}\label{remintg}
 The converse  to Theorem \ref{integrable} is not true in general. 
 First of all, we need enough ``space'' to have a canonical $G$-basis, so we assumme 
 \begin{equation}\label{degree}
  [K:C_K]=p^e
 \end{equation}
 for an integrable $G[m]$-field $(K,\mathbb{D})$.
 For example $D_{\mathbf{i}}=0$, for all $\mathbf{i}\neq\mathbf{0}$, is integrable, but $[K:C_K]=1$
 and there are not enough $p$-independent elements to form a $G$-basis. Equality (\ref{degree}) in dimension $e=1$ means
 that $D_1\neq 0$ and such an assumption is needed in \cite[Theorem 27.3.(ii)]{mat} to obtain the existence of a one-element
 canonical basis. Hence it is morally justified to assume (\ref{degree}) in the next Section,
 where we find a canonical $G$-basis for a special algebraic group $G$.
 Perhaps there are no general reasons for the converse theorem to hold and finding a canonical $G$-basis is the only possibility
 for proving the existence of such a basis
 for a given algebraic group $G$.
\end{remark}

\section{New examples of groups with canonical $G$-bases}\label{secnewex}

\subsection{Unipotent groups of dimension two}\label{subnewex1}
In this subsection, we are going to find a canonical $G$-basis for an algebraic group $G$ of a special type. Firstly we
provide a well known fact about derivations, 
then define $G$ and its group law. After this we specify which tuples satisfy the canonical $G$-basis condition in this case and prove the existence
of such basis for a $G[m]$-field $(K,\mathbb{D})$ satisfying $[K:C_K]=p^e$.

\begin{fact}\label{ZM}
 Let $L$ be a field, $\partial\in\Der_C(L)$, $\ker\partial=C\neq L$, $\partial^{(p)}=0$, then
 \begin{itemize}
  \item[i)] there exists an element $z\in L$ such that
  $\partial(z)=1$, and $1,z,z^2,\ldots,z^{p-1}$ form a basis of $L$ over $C$,
  \item[ii)] $\ker\partial^{(p-1)}=\im\partial$.
 \end{itemize}
\end{fact}

\begin{proof}
 The first item is contained in \cite[Theorem 27.3]{mat}. The second item is in \cite[Lemma 3.]{Zieg2}, 
 but for reader's convenience, we include a short proof. The derivation $\partial$ is a $C$-linear map, after computing $\partial$
 on $1,z,z^2,\ldots,z^{p-1}$ we see that $\dim_C\ker\partial=1$. Therefore $\dim_C\im\partial=p-1$, moreover $\dim_C\ker\partial^{(p-1)}\le p-1$.
 The condition $\partial^{(p)}=0$ implies that $\im\partial\subseteq\ker\partial^{(p-1)}$, but $\dim_C\ker\partial^{(p-1)}\le \dim_C\im\partial$.
\end{proof}

\noindent
For $i\le p$ let $\lambda_i:=\frac{(p-1)!}{(p-i)!i!}\mod p$, which is equal to
the image of $\frac{1}{p}{p \choose i}$ in $\mathbb{F}_p$.
Following the page 171. from \cite{serre1988algebraic}, we define
$$H_n(X_2,Y_2):=\left[\frac{1}{p}\bigg( (X_2+Y_2)^p-X_2^p-Y_2^p\bigg)\right]^{p^n}=$$ $$=\frac{1}{p}\bigg( (X_2^{p^n}+Y_2^{p^n})^p - X_2^{p^{n+1}} 
- Y_2^{p^{n+1}}\bigg)\in\mathbb{F}_p[X_2,Y_2],$$
$$H_n(X_2,Y_2)=\sum\limits_{i=1}^{p-1}\lambda_i X_2^{i p^n} Y_2^{(p-i) p^n}.$$
Consider the  extension of commutative algebraic groups
$$0\to \mathbb{G}_a \to G \to \mathbb{G}_a \to 0,$$
where the group operation $*$ on $G$ is given by
$$(X_1,X_2)*(Y_1,Y_2)=F(X_1,X_2,Y_1,Y_2):=(X_1+Y_1+  \sum\limits_{n=0}^{M}\alpha_n H_n(X_2,Y_2), X_2+Y_2),$$
for a fixed $M\in\mathbb{N}$ and $\alpha_i\in k$ for $i\le M$.
We are interested in the following $m$-truncation
$$F[m](v_1,v_2,w_1,w_2)=(v_1+w_1+  \sum\limits_{n=0}^{N}\alpha_n H_n(v_2,w_2), v_2+w_2),$$
where $v_1$, $v_2$, $w_1$ and $w_2$ are $m$-truncated variables and $N:=\min\lbrace M,m-1\rbrace$.
Without loss of generality we will assume that $N=m-1$.
Let $(K,\mathbb{D})$ be a $G[m]$-field
(i.e. $(K,\mathbb{D})$ is a $\hat{G}[m]$-field,
but $\hat{G}=\hat{F}=F$, so we consider just an $F[m]$-field),
such that $[K:C_K]=p^2$.
%We assume that $\alpha_N\neq 0$.
%It is not hard to prove the fact below.

\begin{lemma}\label{U-simple-comp}
We have the following
\begin{itemize}
 \item[i)] $D_{(i,j)}=D_{(i,0)}D_{(0,j)}=D_{(0,j)}D_{(i,0)}$,
 \item[ii)]$D_{(i_2,0)}\circ D_{(i_1,0)}={i_1+i_2\choose i_1}D_{(i_1+i_2,0)}$.
%\begin{fact}\label{U-p-comp}
 \item[iii)] $[p]_{F[m]}(v_1,v_2)=(-\sum\limits_{n=0}^{N}\alpha_n v_2^{p^{n+1}},0)$,
 \item[iv)] $D_{(i,j)}^{(p)}=0$ for every $i,j\le p^{m-1}$ such that $i\neq 0$,
 \item[v)] if $j<m$ then $$D_{(0,p^j)}^{(p)}=-\alpha_{j}D_{(1,0)}+\sum\limits_{n=p}^{p^j}\beta_n D_{(n,0)},$$
 for some $\beta_n\in k$.
\end{itemize}
\end{lemma}

\begin{proof}
 The first two items are easy.
 For the proof of third item it is sufficient to prove inductively the following
 $$[l]_F(v_1,v_2)=\big(lv_1+\sum\limits_{n=0}^{N}\alpha_n \big[\big(\frac{1}{p}(l^p-l)\big)v_2^p\big]^{p^n},lv_2\big).$$
 The fourth and the fifth item use the third part and Corollary \ref{evP}. Specifically, one needs to show
 $$D_{(0,j)}^{(p)} 
 =\sum\limits_{i_0+i_1p+\ldots+i_Np^N=j} (-1)^{i_0+\ldots+i_N}\cdot\frac{(i_0+\ldots+i_N)!}{i_0!\cdot\ldots\cdot i_N!}
 \;\;\alpha_0^{i_0}\cdot\ldots\cdot\alpha_N^{i_N}D_{(i_0+\ldots+i_N,0)},$$
 for every $j\le p^{m-1}$. We leave it to the reader.
\end{proof}

Let us consider the canonical $F$-derivation from Example \ref{exDF}:
$$\ev_F:k\llbracket X_1,X_2\rrbracket\to (k\llbracket X_1,X_2\rrbracket)\llbracket Y_1,Y_2\rrbracket,$$
where
$$X_1\mapsto X_1+Y_1+
\sum\limits_{n=0}^{N}\alpha_n\sum\limits_{i=1}^{p-1}\lambda_i X_2^{ip^n} Y_2^{(p-i)p^n},$$
$$X_2\mapsto X_2+Y_2.$$
As in Example \ref{exDG}, the above $F$-derivation could be considered as a $G$-derivation on $k(G)$ (because $F=\hat{G}$).
In this situation $k(G)=k(X_1,X_2)$, so we need to find an embedding $\varphi$ of $( k(X_1,X_2), \mathbb{D}^G[m] )$ in $(K,\mathbb{D})$ such that
for $x=\varphi(X_1)$ and $y=\varphi(X_2)$ we have 

\begin{itemize}
 \item[i)] $[C_K(x,y):C_K]=p^2$,
 \item[ii)] $\sum\limits_{i,j=0}^{p^m-1}D_{(i,j)}(x)v_1^iv_2^j=x+v_1+\sum\limits_{n=0}^{N}\alpha_n\sum\limits_{i=1}^{p-1}\lambda_i y^{ip^n} v_2^{(p-i)p^n},$
 \item[iii)] $\sum\limits_{i,j=0}^{p^m-1}D_{(i,j)}(y)v_1^iv_2^j=y+v_2.$
\end{itemize}
The conditions i), ii) and iii) above are equivalent to
$$D_{(1,0)}(x)=1,\quad D_{(0,1)}(x)=\alpha_0\lambda_1 y^{p-1},\quad\ldots,$$
$$D_{(p^n,0)}(x)=0,\quad D_{(0,p^n)}(x)=\alpha_n\lambda_1y^{(p-1)p^n},\;\ldots$$
$$D_{(1,0)}(y)=0,\quad, D_{(0,1)}(y)=1,\quad D_{(p,0)}(y)=0,\quad D_{(0,p)}(y)=0,\;\ldots$$
(since $D_{(1,0)}(x)=1$, $D_{(1,0)}(y)=0$ and $D_{(0,1)}(y)=1$ imply, by Fact \ref{kombajn}, that\\ $[C_K(x,y):C_K]=p^2$).
We are concerned now only with the terms of the form $D_{(p^i,0)}$ and $D_{(0,p^j)}$. It will turn out later that it is enough to consider such terms
to obtain expected $G$-basis.
Recall that $G$ is commutative, so each subset of constants is preserved, i.e.:
$$D_{(i,j)}\big(C_{(i',j')}\big)\subseteq C_{(i',j')},$$
for every $i,j,i',j'<p^m$. 

Recall also that $[F_{s-1}:F_s]=p^2$ for every $s\in\lbrace 0,\ldots,m-1\rbrace$ (Remark \ref{remclaim1}).

\begin{fact}\label{xy1}
 There exists $x,y\in K$ such that $D_{(1,0)}(x)=1$, $D_{(1,0)}(y)=0$ and $D_{(0,1)}(y)=1$.
\end{fact}

\begin{proof}
 Note that $D_{(1,0)}\in\Der_{C_{(1,0)}}(F_{-1})$ and $D_{(1,0)}^{(p)}=0$, so we can use Lemma \ref{dimconstants} and obtain
 $[F_{-1}:C_{(1,0)}]\le p$ (Lemma \ref{dimconstants} works for iterative HS-derivations, but by
 \cite[Theorem 27.4]{mat}, $D_{(1,0)}^{(p)}=0$ implies $\mathbb{G}_a$-iterativity). Moreover, from Lemma \ref{U-simple-comp}, 
 $$D_{(0,1)}^{(p)}|_{C_{(1,0)}}=-\alpha_0 D_{(1,0)}|_{C_{(1,0)}}=0,$$
 so similarly $[C_{(1,0)}:F_0]\le p$. Since $[F_{-1}:F_0]=[K:C_K]=p^2$, both $D_{(1,0)}\in\Der_{C_{(1,0)}}(F_{-1})$
 and $D_{(0,1)}|_{C_{(1,0)}}\in\Der_{F_0}(C_{(1,0)})$ are non-zero, so they satisfy assumptions of Fact \ref{ZM}.i).
\end{proof}

\begin{lemma}\label{Fkernels}
 Let $n\ge 0$ and $i,j<p^{n+1}$, be such that $(i,j)\neq (0,0)$. Then we have $F_n\subseteq C_{(i,j)}$.
\end{lemma}

\begin{proof}
We argue inductively on $l=i+j$ to show that $D_{(i,j)}|_{F_n}=0$ for $i,j<p^{n+1}$ such that $(i,j)\neq (0,0)$.
For $l=1$ it is clear. Assume that $i,j<p^{n+1}$, $(i,j)\neq (0,0)$ and for every $i'+j'<i+j$ such that $(i',j')\neq (0,0)$ we have $D_{(i',j')}|_{F_n}=0$.
If $i=\gamma_0+\ldots+\gamma_r p^r$, $j=\beta_0+\ldots+\beta_s p^s$, $\;\,r,s\le n$, $0\le \gamma_0,\ldots,\gamma_r,
\beta_0,\ldots,\beta_s<p$, and $\gamma_r,\beta_s\neq 0$, then from Fact \ref{Newton-comp}
$$\gamma_r\beta_s D_{(i,j)}=D_{(i-p^r,j-p^s)}\circ 
D_{(p^r,p^s)}- \mathbf{r}(D_{(i',j')})_{0<i'+j'<i+j}.$$
By Lemma \ref{U-simple-comp} $ F_n\subseteq\ker D_{(p^r,0)}\subseteq \ker D_{(p^r, p^s)}$.
\end{proof}

\begin{lemma}\label{Fsubfields}
 For each $n\ge -1$ sets $F_n$ and $F_n\cap C_{(p^{n+1},0)}$ are subfields of $K$.
\end{lemma}

\begin{proof}
By Remark \ref{Fs} $F_{n}$ is a subfield.
Using Lemma \ref{Fkernels}, we get that $D_{(p^{n+1},0)}|_{F_n}\in\Der(F_n)$
and therefore $\ker D_{(p^{n+1},0)}|_{F_n}$, equal to $F_n\cap C_{(p^{n+1},0)}$, also is a subfield.
\end{proof}

\subsubsection{Finding $y$}

\begin{fact}\label{y1}
There exists an element $y\in K$ such that
\begin{itemize}
 \item[i)] $D_{(1,0)}(y)=0$ and $D_{(0,1)}(y)=1$,
 \item[ii)] for every $0<n<m$ we have $D_{(p^n,0)}(y)=D_{(0,p^n)}(y)=0$.
\end{itemize}
\end{fact}

\begin{proof}
 For the proof of i) consider the element $y\in K$ from Fact \ref{xy1}.
 For the proof of ii) 
 we will inductively correct $y$. Take the maximal $0< l < m$ such that  for every $0<l'<l$
  $$D_{(p^{l'},0)}(y)=D_{(0,p^{l'})}(y)=0.$$
 Assume that $D_{(p^l,0)}(y)\neq 0$ or $D_{(0,p^l)}(y)\neq 0$, otherwise we have nothing to do.
 
 \noindent
 Let $D_{(p^l,0)}(y)\neq 0$, clearly $D_{(p^l,0)}(y)\in F_{l-1}$, so
 $D_{(p^l,0)}(y)\in\ker D_{(p^l,0)}|_{F_{l-1}}^{(p-1)}$ equal, due to Fact \ref{ZM}, to the image of
 $D_{(p^l,0)}|_{F_{l-1}}$. There exists $z\in F_{l-1}$ such that
 $D_{(p^l,0)}(y)=D_{(p^l,0)}(z)$. We exchange $y$ with $y-z$.

 \noindent
 Now let $D_{(p^l,0)}(y)=0$
 and $D_{(0,p^l)}(y)\neq 0$. We have $D_{(0,p^l)}(y)\in F_{l-1}\cap C_{(p^l,0)}$ and as before
 $D_{(0,p^l)}(y)\in\ker D_{(0,p^l)}|_{F_{l-1}\cap C_{(p^l,0)}}^{(p-1)}$. Again, we would like to use Fact \ref{ZM}, so it is enough to chceck
 that $D_{(0,p^l)}|_{F_{l-1}\cap C_{(p^l,0)}}^{(p)}=0$, which follows from Lemma \ref{U-simple-comp} and Lemma \ref{Fkernels}.
\end{proof}
 
\noindent
For the rest of this subsection we fix $y\in K$ as in the fact above.

\begin{remark}\label{claim-rem1}
 If $p^{q}\le n<p^{q+1}$ and $D_{(p^{q},0)}(a)=0$, then also $D_{(n,0)}(a)=0$.
 \end{remark}
 
\begin{proof}
It is a property of the standard iterativity rule.
\end{proof}
The values of $D_{(p^n,0)}$ and $D_{(0,p^n)}$ ($n<m$) at the element $y$ determine the value of $D_{(i,j)}$ (for every $(i,j)$) at $y$,
which we show below. Moreover, the proposition below assures us that $y$ fullfils the canonical $G$-basis conditions.

\begin{prop}\label{y-claim}
We have the following
 \begin{itemize}
  \item[i)] for all $n>0$ $D_{(n,0)}(y)=0$,
  \item[ii)] $D_{(n,0)}(y^s)=0$ for all $n>0$ and $1\le s\le p-1$,
  \item[iii)]for all $n>1$ $D_{(0,n)}(y)=0$,
  \item[iv)] $D_{(0,p^n)}(y^s)=0$ for all $n>1$ and $1\le s\le p-1$,
  \item[v)]
  $$D_{(i,j)}(y)=\begin{cases} y &\mbox{if } (i,j)=(0,0),
			      \\ 1 &\mbox{if } (i,j)=(0,1),
			      \\ 0 &\mbox{otherwise.}
                 \end{cases}$$
 \end{itemize}
\end{prop}

\begin{proof}
 The item i) follows from Remark \ref{claim-rem1}. The item ii) is a consequence of the equality $D_{(n,0)}(y^s)=yD_{(n,0)}(y^{s-1})$.
 Our iterativity rule forces (by Fact \ref{Newton-comp}) that
 
 \begin{equation}\label{y-comp}
  D_{(0,j_2)}D_{(0,j_1)}={j_1+j_2\choose j_1}D_{(0,j_1+j_2)}+ \mathbf{r}(D_{(i,j)})_{\substack{0<i+j<j_1+j_2 \\  i\neq 0}}.
 \end{equation}
 For the proof of item iii) we use the equation above in an induction argument. If $p=2$, then $D_{(0,2)}(y)=0$. For $p>2$ we have
 $$0=D_{(0,1)}D_{(0,1)}(y)=2D_{(0,2)}(y)+\mathbf{r}(D_{(0,j)}D_{(i,0)})_{i\neq 0}(y)=2D_{(0,2)}(y).$$
 Assume that $n\ge 2$ and $D_{(0,2)}(y)=\ldots=D_{(0,n)}(y)=0$. Take $n+1=\gamma_0+\gamma_1p+\ldots+\gamma_sp^s$, where
 $\gamma_0,\ldots,\gamma_s<p$, $\gamma_s\neq 0$.
 $$D_{(0,n+1-p^s)}D_{(0,p^s)}(y)=\gamma_s D_{(0,n+1)}(y)+
 \mathbf{r}(D_{(0,j)}D_{(i,0)})_{i\neq 0}(y)=\gamma_s D_{(0,n+1)}(y).$$
 If $s=0$, then the left-hand side of the last expression is equal to $D_{(0,n)}D_{(0,1)}(y)=0$,
 if $s\neq 0$ we proceed similarly due to the equation $D_{(0,p^s)}(y)=0$. The proof of the item iv)
 uses the equation 
 $$D_{(0,l)}(y^s)=yD_{(0,l)}(y^{s-1})+D_{(0,l-1)}(y^{s-1})$$
 and it is a simple induction on $s$. The item v)
 follows from Lemma \ref{U-simple-comp}.i).
\end{proof}

\subsubsection{Finding $x$}

\begin{fact}\label{w}
There exists an element $w\in K$ such that
\begin{itemize}
 \item[i)] $D_{(1,0)}(w)=1$ and $D_{(0,1)}(w)=0$,
 \item[ii)] for every $0<n<m$ we have $D_{(p^n,0)}(w)=D_{(0,p^n)}(w)=0$.
\end{itemize}
\end{fact}

\begin{proof}
 We define $D_{(0,1)}^{\ast}:=D_{(0,1)}|_{C_{(1,0)}}$, note that $D_{(0,1)}^{\ast(p)}=0$ and $D_{(0,1)}^{\ast}\neq 0$. 
 We start with an element $x$ from the statement of Fact \ref{xy1}, for which we have $D_{(0,1)}(x)\in C_{(1,0)}$. Naturally
 $D_{(0,1)}(x)\in\ker D_{(0,1)}^{\ast(p-1)}=\im D_{(0,1)}^{\ast}$. Hence there exists an element $z\in C_{(1,0)}$
 such that $D_{(0,1)}(x)=D_{(0,1)}(z)$. Taking $w=x-z$ give us the first part. The second part follows as in the proof of Fact \ref{y1}.
\end{proof}

\begin{lemma} There exists an element $x\in K$ satisfying
\begin{itemize}
 \item[i)] $D_{(1,0)}(x)=1,\;\; D_{(0,1)}(x)=\alpha_0 y^{p-1}$,
 \item[ii)] $D_{(p^n,0)}(x)=0$ and $D_{(0,p^n)}(x)=\alpha_n y^{(p-1)p^n}$ for each $0<n\le N$,
 \item[iii)]$D_{(p^n,0)}(x)=D_{(0,p^n)}(x)=0$ for each $N<n<m$.
\end{itemize}
\end{lemma}
\begin{proof}
The proof of item ii) is more complicated, but reasoning is similar to the proof of the point i).
\newline
i) We start with $x\in K$ from Fact \ref{xy1}. If $\alpha_0=0$, we proceed like in the proof of Fact \ref{y1}. Assume $\alpha_0\neq0$,
we need $x'\in C_{(1,0)}$ such that $D_{(0,1)}(x+x')=\alpha_0y^{p-1}$. We have
$$D_{(1,0)}\big(\alpha_0y^{p-1}-D_{(0,1)}(x)\big)=\alpha_0 D_{(1,0)}(y^{p-1})=0,$$
therefore $\alpha_0y^{p-1}-D_{(0,1)}(x)\in\ker D_{(1,0)}\subseteq \ker D_{(1,0)}^{(p-1)}=\im D_{(1,0)}$.
So there exists $z\in K$ such that
$$\alpha_0y^{p-1}-D_{(0,1)}(x)=D_{(1,0)}(z)=-\frac{1}{\alpha_0}D_{(0,1)}^{(p)}(z).$$
The last equality comes from Lemma \ref{U-simple-comp}, since $D_{(0,1)}^{(p)}=-\alpha_0D_{(1,0)}$. Then we can take
$x'=-\frac{1}{\alpha_0}D_{(0,1)}^{(p-1)}(z)$, since:
%the substitution $x\mapsto x+x'$ works fine, for example
$$D_{(1,0)}\big(x-\frac{1}{\alpha_0}D_{(0,1)}^{(p-1)}(z)\big)=1-\frac{1}{\alpha_0}D_{(0,1)}^{(p-1)}\big(D_{(1,0)}(z)\big),$$
but
\begin{IEEEeqnarray}{rCl}\label{eq001}
D_{(0,1)}^{(p-1)}\big(D_{(1,0)}(z)\big) &=& D_{(0,1)}^{(p-1)}\big( \alpha_0y^{p-1}-D_{(0,1)}(x) \big)\\
&=&\alpha_0 D_{(0,1)}^{(p-1)}(y^{p-1})-
D_{(0,1)}^{(p)}(x) \nonumber\\
&=&\alpha_0(p-1)!+\alpha_0 D_{(1,0)}(x)=\alpha_0\big( (p-1)!+1\big)\nonumber
\end{IEEEeqnarray}
and by Wilson's theorem it is equal to $0$.
\
\\
\\
ii) As in the proof of \ref{y1}, we take the maximal $0<l\le N$ such that for every $0<l'<l$
$$D_{(p^{l'},0)}(x)=0,\;\; D_{(0,p^{l'})}(x)=\alpha_{l'} y^{(p-1)p^{l'}}.$$
Case 1., $D_{(p^l,0)}(x)\neq 0$. 
\newline 
Clearly, $D_{(p^l,0)}(x)\in C_{(p^{l'},0)}$ for every $0\le l'<l$. Moreover
for every $0 \le l'<l$

\begin{IEEEeqnarray}{rCl}\label{eq002}
D_{(0,p^{l'})}D_{(p^{l},0)}(x) &=& D_{(p^{l},0)}D_{(0,p^{l'})}(x) \\
 &=& D_{(p^{l},0)}\big(\alpha_{l'}  y^{(p-1)p^{l'}}\big) \nonumber \\
 &=& \alpha_{l'}  D_{(p^{l-l'},0)}(y^{p-1})^{p^{l'}}=0, \nonumber
\end{IEEEeqnarray}
by Proposition \ref{y-claim}.ii) where the last equation follows.
This means that $D_{(p^l,0)}(x)\in F_{l-1}$ and furthermore $D_{(p^l,0)}(x)\in \ker D_{(p^l,0)}|_{F_{l-1}}^{(p-1)}=\im D_{(p^l,0)}|_{F_{l-1}}$.
Hence there exists $z\in F_{l-1}$ such that $D_{(p^l,0)}(x)=D_{(p^l,0)}(z)$ and we replace $x$ with $x-z$.

\noindent
Case 2., $D_{(p^l,0)}(x)=0$ and $D_{(0,p^l)}(x)\neq\alpha_l  y^{(p-1)p^l}$.
\newline
If $\alpha_l=0$ we argue similarly
as many times before (compare also with the proof of item iii)), so let $\alpha_l\neq 0$. The aim of this part is to find an element $x'\in F_{l-1}\cap C_{(p^l,0)}$ such that
$$D_{(0,p^l)}(x+x')=\alpha_l y^{(p-1)p^l}.$$
We introduce 
$$W:=C_{(0,1)}\cap C_{(p^l,0)} \cap \bigcap\limits_{1\le l'<l} C_{(p^{l'},0)} \cap C_{(0,p^{l'})}.$$
Note that the element $w$ from Fact \ref{w} satisfies $w\in W\backslash\ker D_{(1,0)}^{*}$, where $D_{(1,0)}^{*}:=D_{(1,0)}|_W$.
\\
\textbf{Claim.} $\ker D_{(1,0)}^{*}\subseteq\im D_{(1,0)}^{*}$.

\begin{proof}[Proof of the claim]
Note that $W_0:=W\cap C_{(1,0)}=F_{l-1}\cap C_{(p^l,0)}=\ker D_{(1,0)}^{*}$ is a subfield of $K$ (by Lemma \ref{Fsubfields}). 
Using Lemma \ref{Fkernels} we obtain that $W$ is a vector space over $W_0$. Now take $a\in W$
such that $D_{(1,0)}^{*}(a)=0$, that means $a\in W_0$. The element $a\cdot w$ belongs to $W$ and moreover
$D_{(1,0)}^{*}(a w)=a D_{(1,0)}^{*}(w)=a$, so $a\in\im D_{(1,0)}^{*}$.
\end{proof}
\noindent
It is not to hard to see that
$$\alpha_l y^{(p-1)p^l}-D_{(0,p^l)}(x)\in F_{l-1}.$$
Moreover, since $D_{(1,0)}(y)=0$, we have
$$D_{(p^l,0)}(\alpha_l y^{(p-1)p^l})=\alpha_l D_{(1,0)}(y^{(p-1)})^{p^l}=\alpha_l (p-1)\big(y^{p-2} D_{(1,0)}(y)\big)^{p^l}=0.$$
We conclude that
$$\alpha_l y^{(p-1)p^l}-D_{(0,p^l)}(x)\in F_{l-1}\cap C_{(p^l,0)} = W\cap C_{(1,0)}.$$
In other words $\alpha_l y^{(p-1)p^l}-D_{(0,p^l)}(x)\in\ker D_{(1,0)}^{*}\subseteq\im D_{(1,0)}^{*}$, and there is
$z\in W$ such that
$$\alpha_l y^{(p-1)p^l}-D_{(0,p^l)}(x) = D_{(1,0)}(z).$$
From Lemma \ref{U-simple-comp} we know that
$$D_{(0,p^l)}^{(p)}=-\alpha_{l}D_{(1,0)}+\sum\limits_{n=p}^{p^l}\beta_n D_{(n,0)},$$
for some $\beta_n\in k$.
By Remark \ref{claim-rem1} for every $p\le i\le p^l$ $D_{(i,0)}|_W=0$, consequently $D_{(1,0)}(z)=-\frac{1}{\alpha_l}D_{(0,p^l)}^{(p)}(z)$.
For $x'$ take $-\frac{1}{\alpha_l}D_{(0,p^l)}^{(p-1)}(z)$, only an argument for $D_{(0,p^l)}^{(p-1)}(z)\in C_{(1,0)}$ is missing, and it is straightforward
modification of the equation (\ref{eq001}).
\
\\
\\
iii) It follows the proof of Fact \ref{y1}, we need to check only that $D_{(p^l,0)}(x),D_{(0,p^l)}(x)\in F_{l-1}$ for $l>N$.
Obviously, $D_{(p^l,0)}(x),D_{(0,p^l)}(x)\in \bigcap\limits_{l'=0}^{l-1} C_{(p^{l'},0)}\cap \bigcap\limits_{l'=N+1}^{l-1} C_{(0,p^{l'})}$.
Let $0\le l'\le N$, then
$$D_{(p^l,0)}D_{(0,p^{l'})}(x)=D_{(p^l,0)}\big(\alpha_{l'}y^{(p-1)p^{l'}}\big)=\alpha_{l'}D_{(p^{l-l'},0)}(y^{p-1})^{p^{l'}}=0,$$
as in (\ref{eq002}). Furthermore,
$$D_{(0,p^l)}D_{(0,p^{l'})}(x)=D_{(0,p^l)}\big(\alpha_{l'}y^{(p-1)p^{l'}}\big)=\alpha_{l'} D_{(0,p^{l-l'})}(y^{p-1})^{p^{l'}},$$
so we are done if $D_{(0,p^{l-l'})}(y^{p-1})=0$, which is a part of Proposition \ref{y-claim}.
\end{proof}
We fix $x\in K$ as in the above fact.

\begin{lemma}\label{xi}
We have the following
 \begin{itemize}
  \item[i)] for all $n>1$ $D_{(n,0)}(x)=0$,
  \item[ii)] for all $n>0$
  $$D_{(0,n)}(x)=\begin{cases} \alpha_l\lambda_i y^{(p-i)p^l} &\mbox{if } 0\le l\le N,\;\; 1\le i <p,\;\; n=ip^l,
			      \\ 0 &\mbox{otherwise.}
                 \end{cases}$$
 \end{itemize}
\end{lemma}

\begin{proof}
 The item i) for $n\ge p$ follows from Remark \ref{claim-rem1}, and for $1<n<p$ from Lemma \ref{U-simple-comp} and $D_{(1,0)}(x)=1$.
 To prove the item ii) we argue inductively.
 Note that $\lambda_1=1$, so
 for $n=1$ it is clear. Assume that for every $n'<n$ our thesis is true. Let 
 $n=\gamma_0+\gamma_1p+\ldots+\gamma_sp^s$, where $\gamma_0,\ldots,\gamma_s<p$ and $\gamma_s\neq 0$.
\\
\textbf{Claim 1.} 
 \begin{equation}\label{gam01}
  D_{(0,n-p^s)}D_{(0,p^s)}(x)=\gamma_sD_{(0,n)}(x),
 \end{equation}
 \begin{equation}\label{gam02}
  D_{(0,n-\gamma_sp^s)}D_{(0,\gamma_sp^s)}(x)=D_{(0,n)}(x).
 \end{equation}
\begin{proof}[Proof of the claim 1.]
Both equations have similar proofs, so we consider only the first one.
We start with the equation (\ref{y-comp}) for $j_1=p^s$ and $j_s=n-p^s$:
$$ 
D_{(0,n-p^s)}D_{(0,p^s)}=\gamma_s D_{(0,n)}+ \mathbf{r}(D_{(i,j)})_{\substack{0<i+j<n \\  i\neq 0}}.
$$
Our aim is to show that
$D_{(i,j)}(x)=0$ for $0<i+j<n$, $i\neq 0$.
The component with $D_{(1,0)}$ ($i=1$ and $j=0$) does not occur. To see this, we compare the sides of the equation from the iterativity
definition for our choosen iterativity rule, where on the left-hand side we focus on $D_{(0,n-p^s)}D_{(0,p^s)}$. A non-zero component with $D_{(1,0)}$
implies that $n=p^{s+1}$ and this is impossible.
Let us assume that $i=1$ and $j>0$. Because of $j<n$, $D_{(0,j)}(x)$ is, due to the inductive assumption, equal to $\beta y^r$ for some $\beta\in k$ and $r>0$,
and then $D_{(1,0)}D_{(0,j)}(x)=D_{(1,0)}(\beta y^r)=0$. If $i>1$, then $D_{(i,j)}(x)=D_{(0,j)}D_{(i,0)}(x)=0$.
\end{proof}
\noindent
\textbf{Claim 2.} For every $0\le l\le N$ and $0< i<p$ we have $D_{(0,ip^l)}(x)=\alpha_l\lambda_iy^{(p-i)p^l}$.
\begin{proof}[Proof of the claim 2.]
It is quite an obvious induction, using claim 1.:
\begin{IEEEeqnarray*}{rCl}
D_{(0,(1+i)p^l)}(x) &=& \frac{1}{i+1}D_{(0,p^l)}D_{(0,ip^l)}(x)=\frac{1}{i+1}D_{(0,p^l)}\big( \alpha_l\lambda_iy^{(p-i)p^l} \big) \\
 &=& \frac{\alpha_l\lambda_i}{i+1}D_{(0,1)}(y^{p-i})^{p^l}=\frac{\alpha_l\lambda_i}{i+1}(p-i)y^{(p-i-1)p^l} \\
 &=& \alpha_l\lambda_{i+1}y^{(p-i-1)p^l}.
\end{IEEEeqnarray*}
\end{proof}

\noindent
Now we are going to the proof of the main induction step. We will deal with several cases.
If $s>N$ , then $D_{(0,p^s)}(x)=0$ and the equation (\ref{gam01}) implies that $D_{(0,n)}(x)=0$.
We can assume that $s\le N$ and $n-\gamma_sp^s\neq 0$ (otherwise we apply claim 2), 
$$D_{(0,n)}(x)=D_{(0,n-\gamma_sp^s)}D_{(0,\gamma_sp^s)}(x)=\alpha_s\lambda_{\gamma_s}D_{(0,n-\gamma_sp^s)}\big( (y^{p-\gamma_s})^{p^s} \big).$$
Recall that for every $a\in K$ the element $a^{p^s}$ belongs to $F_{s-1}$, thus by Lemma \ref{Fkernels}
$D_{(0,n-\gamma_sp^s)}(a^{p^s})=0$.
\end{proof}
We show below that fixed element $x$ satisfies the required properties.

\begin{prop}\label{xfinal}
 $$D_{(i,j)}(x)=\begin{cases} x &\mbox{if } (i,j)=(0,0),
			      \\ 1 &\mbox{if } (i,j)=(1,0),
			      \\ \alpha_l\lambda_i y^{(p-i)p^l} &\mbox{if } (i,j)=(0,ip^l), 0\le l\le N, 1\le i<p,
			      \\ 0 &\mbox{otherwise.}
                 \end{cases}$$
\end{prop}

\begin{proof}
 By Lemma \ref{U-simple-comp}.i) we decompose $D_{(i,j)}$ into $D_{(0,j)}D_{(i,0)}$. For $i\ge p$ Remark \ref{claim-rem1} and $D_{(p^n,0)}(x)=0$,
 where $0<n<m$, ensure us that $D_{(i,0)}(x)=0$, thus also $D_{(i,j)}(x)=0$. For $1<i<p$ Remark \ref{claim-rem1}.ii) used in an inductive argument
 give $D_{(i,0)}(x)=0$. If $i=1$, then $D_{(i,j)}(x)=D_{(j,0)}(1)$. Hence $D_{(i,j)}(x)\neq 0$ if and only if $j=0$. The case with $i=0$ is exactly
 Lemma \ref{xi}.
\end{proof}

By propositions \ref{y-claim} and \ref{xfinal} the pair $\lbrace x,y\rbrace$ is a canonical $G$-basis (see the beginning of Section \ref{subnewex1}).
Thus we end with the following:

\begin{cor}
 For $G$ as defined above, any $m\in\mathbb{N}_{>0}$ and any $G[m]$-field $(K,D[m])$ such that $[K:C_K]=p^2$, there is a canonical $G$-basis in $K$.
\end{cor}

\subsection{Canonical $G$-bases for commutative and connected groups}
In the previous subsection we showed the existence of a canonical $G$-basis for every $G[m]$-field $(K,\mathbb{D})$ such that $[K:C_K]=p^e$, where $G$
was very specific. Now we are going to apply those results to a more general class of algebraic groups.

\begin{definition}
 Let $G$ be an algebraic group over $k$
 \begin{enumerate}
  \item We call $G$ \emph{integrable} if for any $m\in\mathbb{N}_{>0}$, every $G[m]$-derivation on a field $K$ such that $[K:C_K]=p^{\dim G}$
 is strongly integrable. 
  \item If for any $m\in\mathbb{N}_{>0}$, every $G[m]$-field $K$ such that $[K:C_K]=p^{\dim G}$ has a canonical $G$-basis, we call $G$ \emph{canonically integrable}.
 \end{enumerate}
\end{definition}
By Theorem \ref{integrable} each canonically integrable algebraic group is integrable.

\begin{lemma}\label{product}
 Let $G$ and $H$ be algebraic groups over $k$. If both are canonically integrable, then also $G\times H$ is canonically integrable.
\end{lemma}

\begin{proof}
 Introduce $A:=G\times H$, $e_1:\dim G$, $e_2:=\dim H$ and let $(K,\mathbb{D})$ be an $A[m]$-field such that $[K:C_K]=p^{e_1+e_2}$.
 We define
 $$\mathbb{D}':=(D'_{(j_1,\ldots,j_{e_1})}:=D_{(j_1,\ldots,j_{e_1},\underbrace{0,\ldots,0}_{e_2\text{ times}})})_{j_1,\ldots,j_{e_1}<p^m},$$
 $$\mathbb{D}^{\prime\prime}:=(D^{\prime\prime}_{(j_{e_1+1},\ldots,j_{e_1+e_2})}:=D_{(\underbrace{0,\ldots,0}_{e_1\text{ times}},j_{e_1+1},\ldots,j_{e_1+e_2})})_{j_{e_1+1},\ldots,j_{e_1+e_2}<p^m}.$$
 From the $A[m]$-iterativity diagram (see Definition \ref{defiterativ})it follows that 
 $$D_{(j_1,\ldots,j_{e_1+e_2})}=D'_{(j_1,\ldots,j_{e_1})}D^{\prime\prime}_{(j_{e_1+1},\ldots,j_{e_1+e_2})},$$
 $\mathbb{D}'$ is $G[m]$-iterative and the second one, $\mathbb{D}^{\prime\prime}$ is $H[m]$-iterative.
 \
 \\
 Taking any $p$-basis of $K$ over $C_K$ and using Remark \ref{remclaim1} assures us that for every $s<m$
 $\;\;[F_{s-1}:F_s]=p^{e_1+e_2}$. Thus $[K:C_K^{\text{abs}}]=p^{m(e_1+e_2)}$.
 For $s<m$ we introduce
 $$F'_s:=\bigcap\limits_{j=0}^{s} C_{(p^j,0,\ldots,0)}\cap\ldots\cap C_{(0,\ldots,0,p^j,\underbrace{0,\ldots,0}_{e_2\text{ times}})}, \qquad
 F'_{-1}:=K,$$
 $$F''_s:=\bigcap\limits_{j=0}^{s} C_{(\underbrace{0,\ldots,0}_{e_1\text{ times}},p^j,0\ldots,0)} \cap\ldots\cap
 C_{(0,\ldots,0,p^j)},\qquad F''_{-1}:=K.$$
 Consider the following tower of subfields
 $$K\supseteq F'_0 \supseteq F'_1 \supseteq \ldots \supseteq F'_{m-1}\supseteq F'_{m-1}\cap F''_0\supseteq F'_{m-1}\cap F''_1\supseteq
 \ldots F'_{m-1}\cap F''_{m-1}.$$
 For every $s<m$, due to Lemma \ref{dimconstants} and $[K:C_K^{\text{abs}}]=p^{m(e_1+e_2)}$, we have
 $$[F'_{s-1}:F'_s]=p^{e_1},\qquad [F'_{m-1}\cap F''_{s-1}: F'_{m-1}\cap F''_s]=p^{e_2}.$$
 In particular
 $$[F'_{m-1}:F'_{m-1}\cap F''_0]=p^{e_2},$$
 so there exists a canonical $H[m]$-basis $\lbrace \beta_1,\ldots, \beta_{e_2}\rbrace$ of $(F'_{m-1},\mathbb{D}'')$.
 Analogously, there exists a canonical $G[m]$-basis $\lbrace b_1,\ldots, b_{e_1}\rbrace$ of $(F''_{m-1},\mathbb{D}')$.
 Elements $\beta_1,\ldots,\beta_{e_2}$ are $p$-independent in $F'_{m-1}$ over $F'_{m-1}\cap F''_0$. By Corollary \ref{lin_disjoint},
 they are also $p$-independent in $K$ over $F''_0$. Similarly for elements $b_1,\ldots, b_{e_1}$, Corollary \ref{lin_disjoint}
 implies that they are $p$-independent in $F''_0$ over $F''_0\cap F'_0$. We have
 $$[F''_0:F''_0\cap F'_0]\le p^{e_1},$$
 hence $F_0(b_1,\ldots,b_{e_1})=F''_0$ (note that $C_K=F_0=F''_0\cap F'_0$). Now we have all the ingredients to state that 
 $B:=\lbrace b_1,\ldots,b_{e_1},\beta_1,\ldots,\beta_{e_2}\rbrace$ is a $p$-basis of $K$ over $C_K$. 
 Verification that $B$ is also a canonical $A$-basis is not hard and left to the reader.
\end{proof}
We note the obvious fact:

\begin{fact}\label{Gisomor}
 Let $G$ and $H$ be isomorphic algebraic groups over $k$. If $G$ is canonically integrable, then also $H$ is canonically integrable.
\end{fact}

We can prove now the main theorem of this paper.

\begin{theorem}\label{thmfinal}
 Let $G$ be a commutative and connected linear algebraic group over an algebraically closed field $k$. If maximal unipotent subgroup of $G$
 has dimension at most 2, then $G$ is integrable.
\end{theorem}

\begin{proof}
 Due to ``Jordan decomposition'' (last theorem on page 70. in \cite{Water}), $G$ decomposes as $G_U\times G_S$, where $G_U$ consists of unipotent elements and $G_S$ of semi-simple elements.
 If the dimension of $G_U$ is equal to 2 we know by \cite[Proposition 8., page 171.]{serre1988algebraic} that $G_U$ is isomorphic to the group
 defined at the beginning of the previous subsection, so it is canonically integrable. If $\dim G_U=1$, then by \cite[Theorem 3.4.9.]{tspringer}
 it is isomorphic to $\mathbb{G}_a$, so canonically integrable by \cite[Proposition 4.5.]{HK1}. We focus now on the semi-simple part.
 By \cite[Lemma 2.4.2.ii)]{tspringer}, $G_S$ is diagonalizable, and by \cite[Corollary 3.2.7.ii)]{tspringer} it is a torus.
 Proposition 4.10 from
 \cite{HK1} states that also $\mathbb{G}_m$ is canonically integrable, so our group $G$ is isomorphic to the product of canonically integrable groups.
 Finally, we use Lemma \ref{product}, Fact \ref{Gisomor} and Theorem \ref{integrable}.
\end{proof}

In most cases of applications of the model theory to the differential algebra, we are dealing with an algebraic group $G$ over a field $k$, which is assumed only to be perfect. One may wonder if Theorem \ref{thmfinal} can be used for such $G[m]$-fields, i.e. for models of $G[m]-\dcf$ (\cite{HK}). The answer is positive, because separable closure of $k$, which is also algebraic closure, is contained in the absolute constants for models of $G[m]-\dcf$ (for an argument check e.g. proof of \cite[Theorem 10.]{Wood1}).

\subsection{Possible generalisations}
The desired generalisation is to drop, in the assumptions of Theorem \ref{thmfinal}, the condition for the dimension of the unipotent component of group $G$. Unfortunately, the ideas from the above proof do not work in the case of unipotent groups of dimension higher than $2$. There are several reasons for that, which will be explained below.

First of all, we are using in Subsection \ref{subnewex1} formulas for the group law of our group $G$. Commutative, connected unipotent groups of dimension $2$ are characterised by \cite[Proposition 8., page 171.]{serre1988algebraic}, so the explicit formulas for the group law are known. For the unipotent groups of dimension $3$ or greater, the best known to the author results coincide with \cite[Theorem 1., page 176.]{serre1988algebraic} and \cite[Theorem 2., page 177.]{serre1988algebraic}. It is unknown how the condition ``being a subgroup" translates to the case of iterative derivations, hence the last reference does not help in finding a canonical basis. However, there is a hope to use \cite[Theorem 1., page 176.]{serre1988algebraic}.
We sketch this idea and reveal difficulties in extending our technique to this context.

Assume that $G$ is isogenous to $W_n$ (the Witt group of dimension $n$). We should give a 
modification of Lemma 2.6 from \cite{HK1}, from which we would conclude that $G$ is integrable if and only if $W_n$ is integrable. If this can be done, then we need to check whether $W_n$ is integrable. Unluckily, the whole procedure from Subsection \ref{subnewex1} can not be extended to show the existence of a canonical basis for $W_n$. Even in the case $p=2$ and $n=3$, some issues appear.
If we translate the group law of $W_3$ (given by e.g. the formulas (a) and (b) in \cite[p. 128]{Witt}) for $p=2$ into the conditions for a canonical basis:
\begin{itemize}
 \item[i)] $[C_K(x,y,z):C_K]=2^3$,
 \item[ii)] $\sum\limits_{i,j,l=0}^{2^m-1}D_{(i,j,l)}(x)v_1^iv_2^jv_3^l=x+v_1-yv_2+yzv_3+zv_2v_3-zv_3^3-z^3v_3,$
 \item[iii)] $\sum\limits_{i,j,l=0}^{2^m-1}D_{(i,j,l)}(y)v_1^iv_2^jv_3^l=y+v_2-zv_3,$
 \item[iv)] $\sum\limits_{i,j,l=0}^{2^m-1}D_{(i,j,l)}(z)v_1^iv_2^jv_3^l=z+v_3,$
\end{itemize}
we can notice occurrence of equations of a new kind:
$$D_{(0,1,1)}(x)=z.$$
Proofs from Subsection \ref{subnewex1} involve only ``one dimensional differential equations" and the above equation is not of such a form. 
The ``one dimensional differential equations" appear, because
after diminishing the dimension by $1$, at the induction step, we deal with one-dimensional subgroup, what is the case for two-dimensional group $G$.

To summarize, generalisations of Theorem \ref{thmfinal} to the higher dimensional unipotent component case need to involve new proofs. It is also possible that such a generalisation can not be done without some additional assumptions, or even can not be done at all.

\bibliographystyle{plain}
\bibliography{moja}

\end{document}